\documentclass[12pt]{amsart}
\usepackage{amsmath, palatino}
\usepackage{amsthm}
\usepackage{enumerate}
\usepackage{amssymb}
\usepackage[utf8]{inputenc}
\usepackage{color}
\usepackage[all,cmtip,arrow,matrix,curve]{xy}
\usepackage{tikz}
\usetikzlibrary{matrix}

\title{Comparison properties of the Cuntz semigroup and applications to C*-algebras}

\author{Joan Bosa}
\address{School of Mathematics and Statistics, University of Glasgow, 15 University Gardens, G12 8QW, Glasgow, UK}
\email{joan.bosa@glasgow.ac.uk}

\author{Henning Petzka}
\address{Mathematisches Institut der WWU Münster, Einsteinstrasse 62, 48149 Münster, Deutschland}
\email{petzka@uni-muenster.de}

\thanks{The first author was partially supported by the grants DGI MICIIN MTM2011-28992-C02-01, MINECO MTM2014-53644-P, the Australian Research Council grant DP150101598 and the Beatriu de Pinos fellow programme (BP-A 00123). The second author was supported by the DFG (SFB 878) }

\theoremstyle{plain}
\newtheorem{lemma}{Lemma}[section]
\newtheorem{theorem}[lemma]{Theorem}

\newtheorem{proposition}[lemma]{Proposition}
\newtheorem{definition}[lemma]{Definition}
\newtheorem*{proposition*}{Proposition}
\newtheorem*{definition*}{Definition}
\newtheorem*{theorem*}{Theorem}

\newtheorem{remark}[lemma]{Remark}
\newtheorem{example}[lemma]{Example}
\newtheorem{question}[lemma]{Question}

\newcommand{\Cs}{{C*-al\-ge\-bra}}

\newcommand{\Cu}{\mathrm{Cu}}

\begin{document}

\begin{abstract}
We study comparison properties in the category $\Cu$ aiming to lift results to the C*-algebraic setting. We introduce a new comparison property and relate it to both the CFP and $\omega$-comparison. We show differences of all properties by providing examples, which suggest that the corona factorization property for C*-algebras might allow for both finite and infinite projections.  In addition, we show that R\o rdam's simple, nuclear C*-algebra with a finite and an infinite projection does not have the CFP.
 \end{abstract}

\maketitle

\section*{Introduction}
Over the last 25 years, the classification of simple, nuclear C*-algebras has inspired a great wealth of research. Recently, a classification has been carried out in two groundbreaking articles \cite{EGLN16,TWW15} for simple C*-algebras of finite nuclear dimension. Nuclear dimension plays the role of a non-commutative covering dimension for nuclear C*-algebras. Requesting this dimension to be finite is one of the strong regularity conditions occurring in the Toms-Winter conjecture, which predicts that three regularity conditions, each with a different flavour, are in fact equivalent. Counterexamples to the conjecture stating that the same classifying invariant, which works in the case of finite nuclear dimension (the so-called Elliott invariant), should work in the general case, appeared in 2003 due to R\o rdam \cite{R03} and in 2008 due to Toms \cite{T08}. The latter exhibited two non-isomorphic AH-algebras that agreed not only on the Elliott invariant, but also on a whole swathe of topological invariants. However, Toms' examples can be distinguished using the Cuntz semigroup $\Cu(\_)$. In this paper, we focus on studying some comparison properties - such as the corona factorization property and weak comparison conditions for the Cuntz semigroup- to capture the structure of some simple C*-algebras.

There are a number of regularity properties dividing those C*-algebras not handled by the classification theorem from \cite{TWW15} into classes of `regular' and `irregular' C*-algebras.  One of them is the corona factorization property (short, CFP), which is a mild regularity property introduced in \cite{EK01} in order to understand the theory of extensions and in particular of when extensions are absorbing \cite{KN06, NgCFP}. Zhang proved that, under the additional assumption of the CFP, there is no simple C*-algebra of real rank zero containing both a finite and an infinite projection. (The same follows from the methods developed in \cite{OPR1}.) In addition to its analytical definition, the CFP has been characterized in \cite{OPR2}, for $\sigma$-unital C*-algebras, as a certain comparison property of the Cuntz semigroup, also called the CFP (for semigroups). Another related comparison property is the $\omega$-comparison, a generalization of the almost unperforation property, which holds in the case of well-behaved C*-algebras (in the sense of the above mentioned classification theorem)(\cite{Robert11}).

The study of comparison and divisibility properties for the Cuntz semigroup was initiated in \cite{BRTTW,OPR1,OPR2, RobertRordam}, where the preceding properties play important roles. In particular, in \cite[Proposition 2.17]{OPR1} it is shown that $\omega$-comparison implies the CFP for Cuntz semigroups; and the converse is left open. In Example \ref{Ex:CFP-noW}, we answer this question negatively providing an abstract Cuntz semigroup that satisfies the CFP, but not $\omega$-comparison. Our abstract semigroup lies in the category $\Cu$ (as defined in \cite{CEI08} and extended by additional axioms from \cite{RorWin10} and \cite{Robert13}) to which the Cuntz semigroup of a C*-algebra naturally belongs. 

%Aside from the two aforementioned comparability conditions for Cuntz semigroups, we introduce two new, natural properties in this paper (called $\beta$-comparison and cancellation of small elements at infinity) which are used to give equivalent definitions of the CFP and $\omega$-comparison. 

All of the comparison properties mentioned above have in common (suitably stated) that they are characterized by those elements in the Cuntz semigroup of a unital C*-algebra, that cannot be represented by the Cuntz equivalence class of a positive element in some matrix algebra over the given algebra, but only appear as the equivalence class of a positive element in the stabilization. If there is a largest element in the Cuntz semigroup (which does exist in the simple case), then we are more precisely concerned about properties of this largest element. In particular, we focus on questions such as : {\it if a multiple of some element $x$ in a Cuntz semigroup equals the largest element, is $x$ itself already the largest element?} Or, {\it if all functionals on the Cuntz semigroup are infinite on an element $x$, must $x$ be the largest element?} Or thirdly, {\it if the sum of two elements $x$ and $y$ equals to the largest element, and $y$ is small in a suitable sense, must $x$ already be equal to this largest element?} 

To this end, we introduce a new comparison property involving the largest element in a $\Cu$- semigroup, which we call $\beta$-comparison, using the concept of an order unit norm as defined in \cite{Good.} as our motivation. We set this new property in relation to the other comparison properties, and we highlight differences with the help of examples.  

In the C*-algebra framework, there are multiple implications which can be deduced from our study of comparison properties. Firstly, Example \ref{Ex:stablyNoFunctionals} suggests that the expected dichotomy (of being either stably finite or purely infinite) in the simple real rank zero case might require an analytical approach: The given examples of abstract $\Cu$-semigroups answer the corresponding question negatively in the algebraic setting, but we don't know whether these examples can be realized as the Cuntz semigroup of a C*-algebra. Secondly, we relate the comparison properties studied in this paper for those $\Cu$-semigroups coming from a C*-algebra. In this setting, we show that for simple C*-algebras $\beta$-comparison and $\omega$-comparison are equivalent properties and that the so-called elementary $\Cu$-semigroups as in \cite{APT14} can not arise as the Cuntz semigroup of a C*-algebra. Finally, in Theorem \ref{RordamsAlgebraNoCFP} we conclude that the only known example of a simple, nuclear C*-algebra with both finite and infinite projections constructed in \cite{R03}, does not have the CFP. 

The outline of the paper is as follows. After fixing notation and recalling some basic facts on Cuntz semigroups in Section \ref{Sec:Preliminaries}, we explore the difference between states and functionals on $\Cu$-semigroups in Section \ref{SectionFunctionals}. The results are used subsequently in the definition of the value $\beta(x,y)$ in Section \ref{SectionBeta} and we explore some of its characteristics. Section \ref{SecCompProp} is focused on comparison properties. We recall all properties relevant for this paper and define both the $\beta$-comparison property, associated to the value $\beta(x,y)$, and cancellation of small elements at infinity. We further show some relations between all described comparison properties and give examples. In Section \ref{SectionApplications}, we apply and expand the results obtained in the algebraic framework to the framework of C*-algebras.

%%%%%%%%%%%%%%%%%%%%%%%%%%%%%%%%%%%%%%%%%%%%%%%%%%%%%%%%%%%%
\section{Notation and Preliminaries}\label{Sec:Preliminaries}
\subsection{Partially ordered Abelian Semigroups}
Throughout, $(W,\leq)$ will denote a partially ordered abelian monoid, i.e., a partially ordered abelian semigroup with neutral element $0$. We shall exclusively be interested in positively ordered semigroups, i.e., semigroups where $0\leq x$ for all $x\in W$. In particular, $\leq$ will extend the algebraic order, that is, if $x+z=y$, then $x\leq y$.

In the following we want to remind  the reader of some commonly used terminology.
An \textbf{order unit} in $W$ is a non-zero element $u$ such that, for all $x\in W$, there is an $n\in\mathbb{N}$ such that $x\leq nu$. 
We define the i\textbf{deal generated by an element} $y$ as $$W_{y}:=\{x\in W \mid \text{ there exists } n\in \mathbb{N}\text{ such that } x\leq ny\}.$$
Given two elements $x,y\in W$, one writes \textbf{$x\propto y$} if $x$ satisfies $x\leq ny$ for some $n\in\mathbb{N}$.

 Given an increasing sequence $(y_n)$ in $W$, an element $y$ is a \textbf{supremum} of $(y_n)$ when it is a least upper bound. When they exist, suprema of increasing sequences are unique, and we will denote them by $\sup (y_n)$. We say that an ordered abelian semigroup $(W,\leq)$ is \textbf{complete} if all the increasing sequences have suprema in $W$.  

 One writes \textbf{$x\ll y$} if, whenever $\{x_{n}\}$ is an increasing sequence for which the supremum exists and satisfies 
$y\leq \sup\,x_{n}$, then $x\leq x_{n}$ for some $n$. An element $x$ is called \textbf{compact}, if $x\ll x$. We write \textbf{$y<_{s}x$} if there exists $k\in \mathbb{N}$ such that $(k+1)y\leq kx$. 

Finally, an element $x$ in $W$ is said to be \textbf{full} if for any $y',y\in W$ with $y'\ll y$, one has $y'\propto x$, denoted by $y\,\bar{\propto}\, x$.
A sequence $\{x_{n}\}$ in $W$ is said to be \textbf{full} if it is increasing and for any $y',y\in W$ with $y'\ll y$, one has $y'\propto x_{n}$ 
for some (hence all sufficiently large) n.
Notice that if $x\in W$ is an order unit, it is also a full element, but the reverse is not true.

We say that $W$ is \textbf{simple} if $x\,\bar{\propto}\, y$ for all nonzero $x,y\in W$. In other words, every nonzero element in a simple semigroup is full.

%%%%%%%%%%%%%%%%%%%%%%%%%%%%%%%%%%%%%%%%%%%%%%%%%%%%%%%%%%%%%%%%%%%%%%%%%%%%

\subsection{$\Cu$-semigroups}
Given a partially ordered abelian monoid $S$, the following axioms were introduced in \cite{CEI08} in order to define a category $\Cu$ of semigroups containing Cuntz semigroup $\Cu(A)$ of any \Cs\,.
\begin{enumerate}[(O1)]
\item Every increasing sequence $(a_n)_{n\in\mathbb{N}}$ in $S$ has a supremum in $S$.
\item Every element $a\in S$ is the supremum of a sequence $(a_n)_n$ such that $a_n\ll a_{n+1}$ for all $n$.
\item If $a,a',b,b'\in S$ satisfy $a'\ll a$ and $b'\ll b$, then $a'+b'\ll a+b$.
\item If $(a_n)_n$ and $(b_n)_n$ are increasing sequences in $S$, then $\sup_n(a_n+b_n)=\sup_n(a_n)+\sup_n(b_n)$. 
\end{enumerate} 
A sequence as in {\rm (O2)} is called \textbf{rapidly increasing}. Moreover, note that, for semigroups in $\Cu$, the order satisfies $x\leq y$ if and only if $x'\leq y$ for all $x'\ll x$.

\begin{definition}
A $\Cu$-semigroup is a partially ordered monoid that satisfies axioms {\rm (O1)-(O4)} from the above paragraph. That is, $S$ is a $\Cu$-semigroup precisely when $S$ lies in the category $\Cu$.
\end{definition}

If $A$ denotes a C*-algebra $A$, then its Cuntz semigroup is the ordered semigroup of Cuntz-equivalence classes of positive elements in the stabilization of $A$, with the direct sum as addition and the order is given by Cuntz-subequivalence.  We refer the reader to the overview article \cite{APT} for the definition of the Cuntz relation and information on the Cuntz semigroup of a C*-algebra. If $A$ denotes a C*-algebra, then its Cuntz semigroup $\Cu(A)$ is a $\Cu$-semigroup (\cite{CEI08}). If $a$ is a positive element in a the stabilization of a C*-algebra $A$, then we denote by $\langle a \rangle$ its equivalence class in $\Cu(A)$. We will also consider the (original) Cuntz semigroup $W(A)$ of equivalence classes of positive elements in matrix algebras over $A$.

There are two additional axioms (O5) and (O6) that have been shown to hold for any $\Cu$-semigroup $S$ coming from a C*-algebra, i.e., for any $S$ such that there is a C*-algebra $A$ with $S=\Cu(A)$ (see \cite{Robert13} for (O5) and \cite{RorWin10} for (O6)). 

\vspace{0.3cm}
\hspace{-0.5cm}{\rm (O5)} (Almost algebraic order) If $x'\ll x \leq y$ in $S$, then there is some $z\in S$ such that $x'+z\leq y\leq x+z$. 

\hspace{-0.5cm}{\rm (O6)} (Almost Riesz decomposition) If $x'\ll x\leq y_1+y_2$ in $S$, then there are elements $x_1\leq x,y_1$ and $x_2\leq x,y_2$ such that $x_1+x_2\geq x'$. 

Therefore, it would be natural to include (O5) and (O6) into the definition of the category $\Cu$ and into the definition of a $\Cu$-semigroup. Since at times we would like to highlight the usage of these additional axioms, we leave the definition of the category $\Cu$ (and $\Cu$-semigroup) as it is and mention explicitily when we assume a given $\Cu$-semigroup to satisfy the additional axioms.

If $S$ is a $\Cu$-semigroup, let us recall that an element $a\in S$ is {\bf finite} if for every element $b\in S$ such that $a+b\leq a$, one has $b=0$. An element is \textbf{infinite} if it is not finite. An infinite element $a\in S$ is {\bf properly infinite} if $2a\leq a$. We say that $S$ is {\bf stably finite} if an element $a\in S$ is finite whenever there exists $\tilde{a}\in S$ with $a\ll \tilde{a}$. In particular, if $S$ contains a largest element, denoted by $\infty$, then the latter condition is equivalent to $a\ll \infty$ (\cite{APT14}). A largest element, $\infty$, always exists whenever $S$ is simple, and it is unique whenever it exists (\cite[Paragraph 5.2.2.]{APT14}). If $S$ is simple, then we say that $S$ is {\bf purely infinite} if $S=\{0,\infty\}$, i.e., if $S$ only contains of the zero element and the largest element.

We finish our preliminary part by recalling that a $\Cu$-semigroup $S$ is said to be {\bf algebraic} if every element $x\in S$ is supremum of a sequence of compact elements, i.e. of elements such that $a\ll a$. A $\Cu$-semigroup $S=\Cu(A)$ coming from a C*-algebra $A$ is algebraic, whenever the underlying C*-algebra $A$ has real rank zero.

%%%%%%%%%%%%%%%%%%%%%%%%%%%%%%%%%%%%%%%%%%%%%%%%%%%%%%%%%%%%%%%%%%%%%%%%%%%%%%%
\section{States and Functionals on $\Cu$-semigroups}\label{SectionFunctionals}

In this section, we recall the notions of an (extended valued) state and a functional on $\Cu$-semigroups, and we will study their interplay. While functionals (by definition) better preserve the order structure of the $\Cu$-semigroup, there are better results on states to conclude how two given elements order-relate. For additional information on states and functionals we refer the reader to \cite[Section 5.2]{APT14}.

For an ordered semigroup $S$ with distinguished element $y\in S$, by an (extended valued) state we mean an ordered semigroup map $f:S\rightarrow [0,\infty]$ such that $f(y)=1$. The set of states on $S$ normalized at $y$ is denoted by $\mathcal{S}(S,y)$. 

By a functional on a $\Cu$-semigroup $S$ we mean an ordered semigroup map that preserves suprema of increasing sequences (which always exist in a $\Cu$-semigroup), and we denote the set of functionals on $S$ by $F(S)$.

If $S$ is a simple $\Cu$-semigroup, then every (nonzero) functional is faithful, but states don't need to be faithful. To understand this difference better, we consider a helpful subsemigroup. If $S$ is a simple $\Cu$-semigroup, then $S$ contains a largest element $\infty$, and we let $$S_{\ll\infty}:=\{s\in S\ |\ s\ll\infty\}.$$

Note that, whenever $S$ is simple, $x\in S_{\ll\infty}$ and $y\in S$ is nonzero, then there is some $n\in\mathbb{N}$ such that $x\leq n\cdot y$ (because $\infty=\sup_n n\cdot y$). This implies for both states and functionals on $S$ alike that, if a state, or functional, is zero on some nonzero element $x\in S_{\ll \infty}$, then it is zero on all of $S_{\ll \infty}$. If this happens for a functional then, as every $x\in S$ can be written as the supremum of a rapidly increasing sequence (in particular, as the supremum of elements in $S_{\ll\infty}$) and because functionals preserve suprema of increasing sequences, the functional must be zero everywhere. That is, nonzero functionals are faithful on simple $\Cu$-semigroups. But for states there is no such condition on suprema of increasing sequences, so a state may very well be nonzero, but zero on $S_{\ll \infty}$. In fact, faithfulness of functionals on simple $\Cu$-semigroups is a consequence of the more general fact that functionals are uniquely determined (also for nonsimple $S$) on $S_{\ll}=\{s\in S|\ \exists t\in S \mbox{ with } s\ll t\}$ (and states are not), which agrees with $S_{\ll\infty}$ whenever $S$ contains a largest element. 

\begin{remark}
If $A$ is a unital C*-algebra, then consider its (original) Cuntz semigroup $W(A)$ of equivalence classes of positive elements in matrix algebras over $A$. For every $x\in W(A)\subseteq \Cu(A)$ we have that $x\leq \langle N\cdot 1_A\rangle \ll \infty$ for some $N$. Hence $W(A)\subseteq \Cu(A)_{\ll}$. Whether the converse holds, a problem that appeared in the literature under the name of  `hereditariness of the Cuntz semigroup', is an open question in general. Positive answers in quite general settings can be found in, e.g., \cite{ABPP} and \cite{BRTTW}.
\end{remark}

Existence of states is connected to infiniteness of the element at which we would like to normalize. This follows from the following result of \cite{Blac.Ror.} on extensions of states on preordered semigroups, which generalizes the well-known corresponding result on extensions of states on ordered groups by Goodearl and Handelman (\cite{GoodHand}):

\begin{theorem}(\cite[Corollary 2.7]{Blac.Ror.})\label{Blac}
Let $(W,\leq,u)$ be a preordered semigroup with $u$ an order-unit, and let $W_{0}$ be a subsemigroup containing $u$ (equipped with the relative preordering).  Then every
state on $W_{0}$ extends to a state on $W$.
\end{theorem}

If $u$ is not an order unit, consider $I(u)$, the order ideal generated by $u$. Then a state on a subsemigroup $W_0$ containing $u$ can be extended to a state $f$ in $\mathcal{S}(I(u),u)$. As we consider extended valued states (i.e., we allow states to take the value $\infty$), we can extend $f$ to all of $W$ by setting
$$\bar{f}(x):= \left \{\begin{array}{lll} f(x) &,\ x\in I(u) \\ \infty   &,\ x\notin I(u) \end{array}\right. $$

Hence, the assumption on $u$ being an order unit can be dropped if one considers extended valued states.

Of course, $\mathcal{S}(S,y)=\emptyset$ whenever some multiple of $y$ is properly infinite. If, on the other hand, no multiple of $y$ is properly infinite, then $f(n\cdot y)=n$ is a well-defined state on $\{0,y,2y,3y,\ldots\}\subseteq S$, which extends to a state on $S$. For later reference, we put this observation into a lemma. 

\begin{lemma}\label{NoStates}
Let $S$ be an ordered semigroup and let $y\in S$. Then the set $\mathcal{S}(S,y)$ of states normalized at $y$ is empty if and only if some multiple of $y$ is properly infinite.
\end{lemma}

For functionals there is no such characterization. Obviously, if $y$ is properly infinite, then $\lambda(y)=\infty$ for all functionals $\lambda$. But it is possible that in a (simple) $\Cu$-semigroup $S$ (with $S=\Cu(A)$ for some C*-algebra $A$) there is some $y\in S$ such that $y$ is infinite on all functionals, while no multiple of $y$ is infinite (see Example \ref{ExampleBetaVsQQ}). Further, we have the following observation:

\begin{lemma}\label{FaithfulStates}
Let $S$ be a simple $\Cu$-semigroup and let $y\in S$. Then $\lambda(y)=\infty$ for all functionals $\lambda\in F(S)$ if and only if there is no faithful state in $\mathcal{S}(S,y)$.
\end{lemma}

\begin{proof}
Suppose $f\in \mathcal{S}(S,y)$ is faithful. Then $\tilde{f}(x):=\sup_{x'\ll x}f(x')$ is a functional (see \cite{Robert13}) with $0<f(y)\leq 1$.  Conversely, as functionals on simple $\Cu$-semigroups are faithful, a suitable scaling of a functional with $\lambda(y)<\infty$  would give a faithful state on $S$.
\end{proof}

Another, related, useful observation is content of the following lemma. Let us denote by $\lambda_\infty$ the functional assigning the value $\infty$ to all (nonzero) $z\in S$.

\begin{lemma}\label{InfiniteMultiples}
Let $S$ be a simple $\Cu$-semigroup. If there is some $x\in S_{\ll\infty}$ such that $\lambda(x)=\infty$ for all $\lambda\in F(S)$, then $\lambda=\lambda_\infty$. In this case, for every (nonzero) $z\in S$, some multiple of $z$ is properly infinite.
\end{lemma}

\begin{proof}
The first statement is clear as for any two $x,y \in S_{\ll\infty}$ there is $m\in\mathbb{N}$ such that $x\leq m\cdot y$. For the second statement, pick (nonzero) $z\in S$ and suppose that there is a state $f\in \mathcal{S}(S,z)$. Then $\tilde{f}(x):=\sup_{x'\ll x}f(x')$ is a nonzero functional (\cite{Robert13}) with finite values, which contradicts the assumption. By Lemma \ref{NoStates}, some multiple of $z$ is properly infinite. 
\end{proof}

By \cite[Proposition 2.1]{OPR2} (which also follows from Goodearl and Handelman's paper \cite{GoodHand}), for elements $x,y$ in an ordered semigroup $S$, $x<_s y$  is equivalent to the statement that $f(x)< f(y)=1$ for all states $f$ normalized at $y$. It is not known to the authors whether, in the case that $S$ is a $\Cu$-semigroup, this is equivalent to the statement that $\lambda (x)< \lambda (y)=1$ for all functionals $\lambda$ normalized at $y$. A slightly weaker statement was shown by Robert:

%%%  OLD VERSION   %%%%%%%%%%%%%%%%%%%%%%

%\begin{lemma}(\cite{Robert13})\label{ComparisonFunctionals}
%Let $S$ be a simple $\Cu$-semigroup, let $x,y\in S$ such that $y\ll \infty$, and suppose that $\lambda(x)<\lambda(y)$ for all functionals $\lambda\in F(S)$ with $\lambda(y)<\infty$, then %for all $x'\ll x$ we have $x'<_s y$.
%\end{lemma}
%
%\begin{proof}
%Suppose first that $\lambda(y)=\infty$ for all functionals $\lambda$. Then as $S$ is simple and $y\ll \infty$, $\lambda(x)=\infty$ for any nonzero $x\in S$. Then there also exists no state %$f$  on $S$ normalized at $y$, since otherwise $\tilde{f}(z):=\sup_{z'\ll z}f(z')$ is a nonzero functional with finite values. By Lemma \ref{prop1}, some multiple of $y$ is infinite. In %particular, there is $n\in\mathbb{N}$ such that $(n+1) x\leq \infty =ny$.
%
%Suppose then that there exists at least one functional normalized at $y$ and that for all such functionals $\lambda$ we have the strict inequality $\lambda(x)<\lambda(y)$. Let $f$ be a %state normalized at $y$. Then, for all $x'\ll x$, we get
%$$f(x')\leq \sup_{z\ll x} f(z)=\tilde{f}(x)<\tilde{f}(y)\leq f(y).$$
%Hence $f(x)<f(y)$ for all states $f$ normalized at $y$, and \cite[Proposition 2.1]{OPR2} shows that $x<_s y$.
%\end{proof}

%%%%%%%%%%%%%%%%%%%%%%%%%%%%%%%%%%%%%%%%%%

\begin{lemma}(\cite{Robert13})\label{ComparisonFunctionals}
Let $S$ be a simple $\Cu$-semigroup, let $x,y\in S$. Suppose that $\lambda(x)<\lambda(y)$ for all functionals $\lambda\in F(S)$ finite on $S_{\ll\infty}$. Then for all $x'\ll x$ we have $x'<_s y$.
\end{lemma}

\begin{proof}
Let $(y_n)$ be a rapidly increasing nonzero sequence in $S$ with supremum $y$. 

Suppose first that there is no functional that is finite on $S_{\ll\infty}$. Then, by Lemma \ref{InfiniteMultiples}, some multiple of each $y_n$ is infinite. In particular, there is $k\in\mathbb{N}$ such that $(k+1) x\leq \infty =ky$, implying that $x'<_sy$.

Hence, from now on we may assume that for every $n\in \mathbb{N}$ there exists at least one functional normalized at $y_n$. As functionals preserve suprema of increasing sequences, and since $F(S)$ is compact (\cite{ERS11}), we find $n\in \mathbb{N}$ so that for all functionals $\lambda$, which are finite on $S_{\ll\infty}$, we have the strict inequality $\lambda(x)<\lambda(y_n)$. Let $f$ be a state normalized at $y_n$. Then $\tilde{f}(x):=\sup_{x'\ll x}f(x')$ is a nonzero functional (\cite{Robert13}). For all $x'\ll x$, we get
$$f(x')\leq \sup_{z\ll x} f(z)=\tilde{f}(x)<\tilde{f}(y_n)\leq f(y_n).$$
Hence, $f(x')<f(y_n)$ for all states $f$ normalized at $y_n$, and \cite[Proposition 2.1]{OPR2} shows that $x'<_s y_n\leq y$.
\end{proof}

From the above we were not able to deduce that $\lambda (x)< \lambda (y)=1$ for all functionals $\lambda$ normalized at $y$ implies that $x <_s y$, but we get a weaker statement, sufficient for our later application.

%%%%%  OLD VERSION  %%%%%%%%%%%%%%%%%%%%%%%%%%%%%%%%%%%%%%
%
%\begin{lemma}\label{UglyWayOut}
%Let $S$ be a simple $\Cu$-semigroup and $x,y \in S$ such that $x\ll \infty$ and $y\ll \infty$. If there is some $m\in \mathbb{R}$ such that $ \lambda(x)\cdot m<\lambda(y)$ for all %functionals $\lambda\in F(S)$ such that $\lambda(y)<\infty$, then $x<_s y$.
%\end{lemma}
%
%\begin{proof}
%Find a strictly increasing sequence $z_n\ll z_{n+1}$ such that $\sup_{n}z_n=\infty$. Then, as $x\ll \infty$, it follows that $x\ll z_n\ll \infty$ for some $n$ and there exists some $m\in %\mathbb{N}$ such that $z_n\leq m\cdot x$. Hence, $x\ll m\cdot x$. Whenever $\lambda(m\cdot x)=m\cdot \lambda(x)<\lambda(y)$ for all functionals with $\lambda(y)<\infty$, we get %that $x<_s y$ by Lemma \ref{ComparisonFunctionals}.
%\end{proof}
%
%%%%%%%%%%%%%%%%%%%%%%%%%%%%%%%%%%%%%%%%%%%%%%%%%%%%%%%

\begin{lemma}\label{UglyWayOut}
Let $S$ be a simple $\Cu$-semigroup and $x\in S$ such that $x\ll \infty$. Then there is some $m\in \mathbb{R}$ such that whenever $y\in S$ satisfies that $\lambda(x)\cdot m<\lambda(y)$ for all functionals $\lambda\in F(S)$ finite on $S_{\ll\infty}$, then $x<_s y$.
\end{lemma}

\begin{proof}
Find a rapidly increasing sequence $z_n\ll z_{n+1}$ such that $\sup_{n}z_n=\infty$. Then, as $x\ll \infty$, it follows that $x\ll z_n\ll \infty$ for some $n$ and there exists some $m\in \mathbb{N}$ such that $z_n\leq m\cdot x$. Hence, $x\ll m\cdot x$. Whenever now $\lambda(m\cdot x)=m\cdot \lambda(x)<\lambda(y)$ for all functionals finite on $S_{\ll\infty}$, then we get that $x<_s y$ by Lemma \ref{ComparisonFunctionals}.
\end{proof}

%%%%%%%%%%%%%%%%%%%%%%%%%%%%%%%%%%%%%%%%%%%%%%%%%%%%%%%%%%%%%%%%%%%
%%%%%%%%%%%%%%%%%%%%%%%%%%%%%%%%%%%%%%%%%%%%%%%%%%%%%%%%%
\section{The value $\beta(x,y)$}\label{SectionBeta}

In this section we explore a value associated to any two elements $x,y$ in a $\Cu$-semigroup $S$, called $\beta(x,y)$. This value is induced by extending the order-unit norm for partially ordered abelian groups as described in (\cite{GoodHand}) to semigroups. In Section \ref{SecCompProp} we will use it to define a comparison property on $\Cu$-semigroups.

Although we will be mainly concerned about $\Cu$-semigroups, we define $\beta(x,y)$ more generally for $x,y$ in an arbitrary ordered abelian semigroup  $(W,\leq)$. 

\begin{definition} (cf. \cite[Section 4]{GoodHand})
Let $(W,\leq)$ be an ordered abelian semigroup and $x,y\in W$ such that $x\propto y$.
 We define the real number $\beta(x,y)$ as $$\beta(x,y)=\inf\{l/k \mid kx\leq ly\text{ where } k,l\in\mathbb{N}\}.$$
\end{definition}

Recall that we denote by $W_y$ the order ideal of $W$ such that for all $x\in W_y$ there exists $n\in \mathbb{N}\text{ such that } x\leq ny$. We provide an equivalent definition of the value $\beta(x,y)$ in the case that $W_y$ allows a state normalized at $y$..

\begin{proposition}\label{equality}
 Let $(W,\leq)$ be an ordered abelian semigroup. If $x\in W_y$ and $\mathcal{S}(W_{y},y)\neq\emptyset$, then
$$\beta(x,y)=\inf\{l/k\mid kx\leq ly\}=\sup\{f(x)\mid f\in \mathcal{S}(W_{y},y)\}.$$ 
\end{proposition}

\begin{proof}
Let us start by defining $W_{0}:=\langle x,y\rangle=\{kx+ly\mid k,l\geq 0\}\subseteq W_y$. Note that the existence of a state $f\in \mathcal{S}(W_{y},y)$ implies the following property:

\begin{equation}\label{EqOne}
\mbox{Whenever } z\in W_0 \mbox{ and } ky+z \leq ly+z\mbox{, then } k\leq l.
\end{equation} 

Consider the map
\[
\begin{tabular}{c c c c}
   $f_{0} : $ & $W_{0}$ &$\longrightarrow$ & $\mathbb{R}^{+}$ \\
       & $kx+ly$& $\mapsto$  & $k\beta(x,y)+l$. \\

 \end{tabular}\\\,
\]
We claim that $f_{0}$ is a state in $\mathcal{S}(W_{0},y)$.
 
As additivity is clear, we subsequently prove that $f_{0}$ is well-defined and that it preserves the order. Namely, given two elements in $W_{0}$, $k_{1}x+l_{1}y$ and $ k_{2}x+l_{2}y$, we must show that $f_{0}(k_{1}x+l_{1}y)\leq f_{0}(k_{2}x+l_{2}y)$ 
if $k_{1}x+l_{1}y\leq k_{2}x+l_{2}y$. We divide the proof in four cases:
\begin{enumerate}[\rm (i)]
 \item $k_{1}\leq k_{2}$ and $l_{1}\leq l_{2}$,
 \item $k_{1}\leq k_{2}$ and $l_{1}\geq l_{2}$,
 \item $k_{1}\geq k_{2}$ and $l_{1}\leq l_{2}$,
 \item $k_{1}> k_{2}$ and $l_{1}> l_{2}$.
\end{enumerate}

We note that {\rm (i)} is trivial and {\rm (iv)} stands in contradiction to Condition (\ref{EqOne}), so let us start showing {\rm (ii)}. In this case, since $k_{1}, k_{2}, l_{1}, l_{2}$ are integers, we can write $k_{2}=k_{1}+k'$ and $l_{1}=l_{2}+l'$ for some $k',l'\in\mathbb{N}$, obtaining $k_{1}x+(l_{2}+l')y\leq (k_{1}+k')x+l_{2}y$. Setting $z:= k_{1}x+l_{2}y$, we get $l'y+z\leq k'x+z\text{ where } z\in W_{0}$. It follows then that $ml'y+z\leq mk'x+z$ for all $m\in \mathbb{N}$.

We have to show that $k_1\beta(x,y) + l_1 \leq k_2\beta(x,y) + l_2$, equivalently $\beta(x,y)\geq l'/k'$. To end up with a contradiction, suppose that $\beta(x,y) < l'/k'$. Then, there exist $a,b\in \mathbb{N}$ such that $\beta(x,y)\leq b/a < l'/k'$ and $ax\leq by$. Then 
$y + k'ax\leq y + k'by \leq al'y.$  Taking $m=a$ in the above equation for the second equality, we get 
$$(al'+1)y+z= aly'+z+y\leq ak'x+z+y\leq al'y+z,$$
which contradicts Condition (\ref{EqOne}). 

Finally to prove {\rm (iii)}, we start as before to obtain $mk''x+z'\leq ml''y+z'$ for all $m$, where $z'=k_{2}x+l_{1}y.$ As we can find $l_0$ such that $z\leq l_0y$, it follows that $mk''x\leq (ml''+l_0)y$ for all $m$ Hence, $\beta(x,y)\leq \inf \{\frac{ml''+l_0}{mk''} |\ m\in\mathbb{N}\}=l''/k''$, which implies that $f_{0}(k_{1}x+l_{1}y)\leq f_{0}(k_{2}x+l_{2}y)$ as desired.

By Theorem \ref{Blac}, $f_{0}$ extends to
 $\mathcal{S}(W_{y},y)$. As all the states on $\mathcal{S}(W_{y},y)$ satisfy that $f(x)\leq \beta(x,y)$ and $f_{0}(x)=\beta(x,y)$,
 we conclude that $\beta(x,y)=\sup\{f(x)\mid f\in \mathcal{S}(W_{y},y)\}$.
\end{proof}

\begin{remark}
\rm{Let $(W,\leq)$ be an ordered abelian semigroup, $x,y\in W$ and $x\in W_{y}$ and $\mathcal{S}(W_y,y)\neq \emptyset$. Then $$\{f(x)\mid f\in \mathcal{S}(W_{y},y)\} = 
\{f(x)\mid f\in \mathcal{S}(W,y)\}.$$ Indeed, let $\phi: \mathcal{S}(W_{y},y)\to \mathcal{S}(W,y)$ be the map that sends $f\mapsto \bar{f}$, where $\bar{f}$ is defined by 
$\bar{f}=f$ on $W_{y}$ and $\infty$ otherwise. Clearly $\bar{f}$ is a state and $\phi$ is well-defined. Notice that $\phi$ is injective. Hence, 
the map $\varphi: \mathcal{S}(W,y)\to \mathcal{S}(W_{y},y)$ defined by $f\mapsto f_{W_{y}}$ is surjective. It follows that $\{f(x)\mid f\in \mathcal{S}(W_{y},y)\}= \{f(x)\mid f\in \mathcal{S}(W,y)\}$, and that
$\beta(x,y)=\inf\{l/k\mid kx\leq ly\}=\sup\{f(x)\mid f\in \mathcal{S}(W,y)\}.$}
\end{remark}

The next lemmas show some properties of the value $\beta(x,y)$.

\begin{lemma}\label{properties}
 Let $(W,\leq)$ be an ordered abelian semigroup and $x,y,z\in W$.
\begin{enumerate}[\rm(i)]
 \item If $x \propto y$ and $y\leq z$ then $\beta(x,y)\geq \beta(x,z)$.
 \item If $y \propto z$ and $x\leq y$ then $\beta(x,z)\leq \beta(y,z)$.
\end{enumerate}
\end{lemma}
\begin{proof}
$\rm(i)$
If $kx\leq ly$, then $kx\leq ly\leq lz$; thus, $$\{l/k\mid kx\leq ly\}\subseteq\{m/n\mid nx\leq mz\}.$$
Therefore, $\inf\{l/k\mid kx\leq ly\}\geq \inf\{m/n\mid nx\leq mz\}.$
       
$\rm(ii)$ If $ky\leq lz$ then $kx\leq ky\leq lz$; thus, $$\{l/k\mid ky\leq lz\}\subseteq \{m/n\mid nx\leq mz\}.$$
Therefore, $\inf\{l/k\mid ky\leq lz\}\geq \inf\{m/n\mid nx\leq mz\}$.
\end{proof}

Note that $\beta(x,y)< 1$ if and only if $x<_s y$. (Recall that $x<_s y$ if $(k+1)\cdot x\leq k\cdot y$ for some $k\in\mathbb{N}$.)

\begin{lemma}\label{properties2}
Let $(W,\leq)$ be an ordered abelian semigroup. Then:
  \begin{enumerate}[\rm(i)]
   \item If $x\in W$ and $\{y_{n}\}$ is a sequence in $W$ satisfying $x\propto y_{j}$ for all $j$, 
	$\beta(x,y_{1}+\ldots+y_{n})\leq (\sum^{n}_{j=1}\beta(x,y_{j})^{-1})^{-1}$,
   \item If $y\in W$ and $\{x_{n}\}$ is a sequence in $W$ satisfying $y\propto x_{j}$ for all $j$,
	$\beta(x_{1}+\ldots+x_{n},y)\leq \sum^{n}_{j=1}\beta(x_{j},y)$,
  \end{enumerate}
\end{lemma}

\begin{proof}
We only prove the second statement, since the first one is shown in a similar fashion. We assume $n=2$ and note that the general case then follows easily. 

Let $\epsilon>0$. For $i=1,2$ find $l_{i},k_{i}\in\mathbb{N}$ such that $\beta(x_{i},y)\leq \frac{l_{i}}{k_{i}}\leq \beta(x_{i},y)+\epsilon$ and $k_{i}x_{i}\leq l_{i}y$. Then
$k_1k_2(x_1+x_2)\leq k_2l_1y + k_1l_2 y$, so $$\beta(x_1+x_2,y)\leq \frac{k_2l_1+k_1l_2}{k_1k_2}=\frac{l_1}{k_1}+\frac{l_2}{k_2}\leq \beta(x_{1},y)+\beta(x_{2},y) +2\epsilon.$$
\end{proof}

\begin{remark}\label{almost_zero}
We would like to emphasize that given $x$ and a sequence $\{y_{n}\}$ belonging to $W$ such that $x<_s y_{j}$ for all $j$, we have:
 
$$\beta(x,y_{1}+y_{2}+\ldots+y_{k})\leq 1/k \text{ for all } k.$$
\end{remark}

\begin{proof}
Assume $k=2$ since it is easy to extend the proof to general $k\in\mathbb{N}$. Consider $y_{1}, y_{2}$ such that $(k_{1}+1)x\leq k_{1}y_{1}$ and $(k_{2}+1)x\leq k_{2}y_{2}$. Using \cite[Proposition 2.1]{OPR2},
there exists $k_{0}\in \mathbb{N}$ such that $(k+1)x\leq ky_{1}$ and $(k+1)x\leq ky_{2}$ for all $k\geq k_{0}$. Adding both inequalities, we obtain 
$ 2kx\leq k(y_{1}+y_{2})$ for all $k\geq k_0$. Thus, $ \beta(x, y_{1}+y_{2})\leq 1/2 .$
\end{proof}

\begin{lemma}\label{prop1}
Let $W$ be an abelian ordered semigroup and $x,y\in W$, such that $y\leq mx$ and $x\leq ny$ for some $m,n\in\mathbb{N}$. Then the following statements are equivalent.
\begin{itemize}
\item[(i)] $\beta(x,y)=0$ 
\item[(ii)] Some multiple of $x$ is properly infinite.
\item[(iii)] Some multiple of $y$ is properly infinite.
\item [(iv)] $\mathcal{S}(W,y)=\emptyset.$
\end{itemize}
\end{lemma}

\begin{proof}
If $\beta(x,y)=0$, then we can find $k,l\in\mathbb{N}$ such that $kx\leq ly$ and $l/k\leq 1/2m$. Then $2kx\leq 2ly \leq 2lm x\leq kx$, hence $kx$ is properly infinite. It then follows further that $2ny\leq 2nmx\leq x\leq ny$, so $y$ is properly infinite. It is easy to see that {\rm (iii)} implies {\rm (i)}. Finally, the equivalence with {\rm (iv)} follows from Lemma \ref{NoStates}.

\end{proof}

\section{Comparison Properties}\label{SecCompProp}

In this section we recall the properties of corona factorization (CFP) and $\omega$-comparison for $\Cu$-semigroups as defined in \cite{OPR2} and show a few equivalent reformulations of each property. We use the value $\beta(x,y)$ from Section 3 to introduce a new comparison notion, called $\beta$-comparison property, which we also prove to be equivalent to $\omega$-comparison for many $\Cu$-semigroups of importance. We demonstrate differences of all mentioned properties with the help of examples. Finally, we discuss the relation between the CFP and the non-existence of stably infinite but finite elements.

Recall that a sequence $(x_n)_n$ in an ordered abelian semigroup $W$ is called full, if it is increasing and for any $z'\ll z$, one has $z'\leq m\cdot x_n$ for some $n,m\in\mathbb{N}.$ For future use note that, if $(x_n)$ is a full sequence in a $\Cu$-semigroup $S$, then $(n\cdot x_n)$ is an increasing sequence in $S$ with supremum equal to the maximal element $\infty$ in $S$. (In particular, a maximal element exists in $S$.) Indeed, for any $t\ll s$ one has $t\leq n \cdot x_n$ for some $n$. Now taking the supremum on both sides (first on the right, then on the left hand side) yields $s\leq \sup_n (n\cdot x_n)$.

We will also want to change a full sequence to a more suitable one in the following way.

\begin{lemma}\label{NewFullSequence}
Let $S$ denote a $\Cu$-semigroup, and let $(x_n)_n$ be a full sequence in $S$. Then there is a full sequence $(x_n')_n$ in $S$ such that $x_n'\ll x_n$ for each $n\in \mathbb{N}$.
\end{lemma}

\begin{proof}
For each $n\in\mathbb{N}$ find a rapidly increasing sequence $(x_n^k)_k$ with supremum $x_n$. Starting with $x_1^1\ll x_1\leq x_2$, we find $k_2$ such that $x_1^1\leq x_2^{k_2}$. Inductively, we find an increasing sequence $(y_n)_n:=(x_n^{k_n})_n$ with $x_i^j\leq x_n^{k_n}\ll x_n$ for each $i,j,n\in \mathbb{N}$ such that $1\leq i,j\leq (n-1)$.  Let us show that the sequence $(y_n)_n$ is full. 

Let $z'\ll z$, and find $z''$ such that $z'\ll z''\ll z$. By assumptions on $(x_n)_n$, we find $m$ and $k$ such that $z'\ll z''\leq m\cdot x_k$. Therefore, there is some $l$ such that $z'\leq m\cdot x_k^l$. We conclude that $$z'\leq m\cdot x_k^l\leq m\cdot x_M^{k_M}=m\cdot y_M,$$ where $M=\max\{k,l\}+1$.
\end{proof}

\begin{definition}(\cite{OPR2})\label{DefCFP}
Let $W$ be an ordered abelian semigroup.
 \begin{itemize}
\item   $W$ satisfies the corona factorization property (CFP) if, given any full sequence $(x_{n})_n$ in $W$, any sequence $(y_{n})_n$ in $W$, an element $x'$ in $W$ such that $x'\ll x_{1}$, and a positive integer $m$ satisfying $x_{n}\leq m y_{n}$ for all $n$, then there
exists a positive integer $k$ such that $x'\leq y_{1}+\ldots+y_{k}$.
 \item $W$ satisfies the strong corona factorization property (StCFP) if, given $x',x\in W$, a sequence $(y_n)$ in $W$, and a positive integer $m$ such that $x'\ll x\leq my_n$ for all $n$, then there exists a positive integer $k$ such that $x'\leq y_1+\ldots+y_k$.
\end{itemize}
\end{definition}

Note that the CFP and the StCFP are equivalent in the simple case. The terminology comes from the fact that a $\sigma$-unital C*-algebra has the corona factorization property if and only if its Cuntz semigroup $W(A)$ has the CFP as defined above (\cite{OPR2}).

The following proposition was basically shown in \cite{RobertRordam}, but our version differs from theirs in that we reduce to elements in $S_{\ll\infty}$. Recall that we denote by $S_{\ll\infty}$ the set of all $y\in S$ such that $y\ll\infty$. 

\begin{proposition}(\cite{RobertRordam})\label{PropCFP}
Let $S$ be a $\Cu$-semigroup. Then the following statements are equivalent:
\begin{itemize}
\item[(i)] $S$ has the CFP.
\item[(ii)] Given any full sequence $(x_{n})_n$ in $S$, any sequence $(y_{n})_n$ in $S_{\ll\infty}$, an element $x'$ in $S$ such that $x'\ll x_{1}$, and a positive integer $m$ satisfying $x_{n}\leq m y_{n}$ for all $n$, then there exists a positive integer $N$ such that $x'\leq y_{1}+\ldots+y_{N}$.
\item[(iii)] Given any full sequence $(x_{n})_n$ in $S$, any sequence $(y_{n})_n$ in $S_{\ll\infty}$ and any positive integer $m$ satisfying $x_{n}\leq m\cdot y_{n}$ for all $n$, then $\infty = \sum_{n=1}^\infty y_{n}$.
\item[(iv)] Given a sequence $(y_n)_n$ in $S_{\ll\infty}$ such that $m\cdot \sum_{n=k}^\infty y_n=\infty$ for some $m$ and all $k\in\mathbb{N}$, then $\sum_{n=1}^\infty y_n=\infty$.
\end{itemize}
\end{proposition}

\begin{proof}
We will show first that, in the definition of the CFP, we can reduce to a sequence $(y_n)_n\in S_{\ll\infty}$, i.e., condition (ii) implies the CFP. So let $(x_{n})_n$ be a full sequence in $S$, $(y_{n})_n$ a sequence in $S$, $x'\ll x_{1}$ in $S$, and $m$ a positive integer $m$ satisfying $x_{n}\leq m y_{n}$ for all $n$. Use Lemma \ref{NewFullSequence} to find a new full sequence $(x'_n)_n$ with $x_n'\ll x_n$, and we may choose the sequence such that $x'\ll x'_1$. Choose for each $n$ a rapidly increasing sequence $(y_n^k)_k$ with supremum $y_n$. Then, for each $n\in\mathbb{N}$, there is $k(n)$ such that $x'_n\leq m\cdot y_n^{k(n)}$.
As each $y_n^k\in S_{\ll\infty}$, we can apply (ii) to find $N$ in $\mathbb{N}$ such that $x'\leq y_1^{k(1)}+\ldots + y_N^{k(N)}\leq y_1+\ldots +y_N$.

The converse, that (i) implies (ii), is trivial.

Let us then suppose statement (ii) and that we are given a full sequence $(x_{n})_n$ in $S$, a sequence $(y_{n})_n$ in $S_{\ll\infty}$ and a positive integer $m$ satisfying $x_{n}\leq m y_{n}$ for all $n$. We choose an injective map $\alpha:\mathbb{N}\times \mathbb{N} \rightarrow \mathbb{N}$. Fix some $k,l \in\mathbb{N}$. Then for any $x'\ll x_{k}$, (ii) gives some $N\in\mathbb{N}$ such that $x'\leq \sum_{n=1}^N y_{\alpha(l,n)}$. Since $x'\ll x_{k}$ was arbitrary, we get, taking supremum on both sides, that $x_{k}\leq \sum_{n=1}^\infty y_{\alpha(l,n)}$. Hence, $$\infty\cdot x_{k}\leq \sum_{l=1}^\infty \sum_{n=1}^\infty y_{\alpha(l,n)}\leq \sum_{n=1}^\infty y_n.$$ The latter holds for arbitrary $k$; thus, $\infty=\sup_k \left ( \infty \cdot x_k\right ) \leq  \sum_{n=1}^\infty y_n$.

Let us suppose (iii) to hold and that we are given a sequence $(y_n)_n$ in $S_{\ll\infty}$ such that $m\cdot\sum_{n=k}^\infty y_n=\infty$ for all $k\in\mathbb{N}$. Pick an arbitrary full sequence $(x_n)_n$ and, with the help of Lemma \ref{NewFullSequence}, find a new full sequence $(x_n')_n$ such that $x_n'\ll x_n$ for all $n$ in $S$. Then, for any $n,k\in \mathbb{N}$ we have that $x_n'\ll x_n\leq \infty = m\cdot \sum_{j=k}^\infty y_j$. Hence, for each $n$ and $k$ there is $N(n,k)\in\mathbb{N}$ such that $x_n'\leq m\cdot \sum_{j=k}^{N(n,k)} y_j$. We choose $z_1:=\sum_{j=1}^{N(1,1)} y_j$, and then, inductively, for given $z_n=\sum_{j=k}^{M} y_j$, we choose $z_{n+1}:=\sum_{j=M+1}^{N(n+1,M+1)} y_j$. (Note that each $z_n\in S_{\ll\infty}$.) It follows that $x_n'\leq m\cdot z_n$ for all $n$, and, by (iii), we obtain that $\infty= \sum_{n=1}^\infty z_n= \sum_{n=1}^\infty y_n$.

Finally, suppose that  (iv) holds and that we are given a full sequence $(x_n)_n$ in $S$, some $x'\ll x_1$, a sequence $(y_n)_n$ in $S_{\ll\infty}$, and a positive integer $m$ such that $ x_n \leq m\cdot y_n$ for each $n$. Then for all $l,k$ such that $l\geq k$, $$m\cdot \sum_{n=k}^\infty y_n\geq \sum_{n=l}^\infty m\cdot y_n \geq \sum_{n=l}^\infty x_n\geq \infty\cdot x_l.$$ Taking supremum over $l$, $m\cdot \sum_{n=k}^\infty y_n=\infty$ for each $k$, which by {\rm(iv)} implies that $\sum_{n=1}^\infty y_n=\infty$. Hence, $x'\ll x \leq \sum_{n=1}^\infty y_n$, which shows the existence of some $k$ such that $x'\leq \sum_{n=1}^k y_n$. Thus, (iv) implies (ii), which completes the proof.

\end{proof}

\begin{definition}
An ordered abelian semigroup $W$  has the $\omega$-comparison property if whenever $x',x,$ $ y_{0}, y_{1}, y_{2},\ldots$ are elements in $W$
such that $x<_{s}y_{j}$ for all $j$ and $x'\ll x$, then $x'\leq y_{0}+y_{1}+\ldots+y_{n}$ for some $n$.
\end{definition}

The following lemma constitutes a reduction step in the proof of Proposition \ref{Omega}.

\begin{lemma}\label{ReductionStepOmega}
Let $S$ be a $\Cu$-semigroup. In the definition of $\omega$-comparison one may, without loss of generality, assume that $y_j\ll\infty$ for all $j$. That is, $\omega$-comparison is equivalent to the following property:
\begin{itemize}
\item  Whenever $x',x,$ $ y_{0}, y_{1}, y_{2},\ldots$ are elements in $S$ such that $x<_{s}y_{j}\ll \infty$ for all $j$ and $x'\ll x$, then $x'\leq y_{0}+y_{1}+\ldots+y_{n}$.
\end{itemize}
\end{lemma}

\begin{proof}
We suppose that we are given a $\Cu$-semigroup $S$ for which we only know the condition of $\omega$-comparison to hold when $y_j\ll\infty$ for all $j$. We will show that then the $\omega$-comparison holds.

Let $x',x,$ $ y_{0}, y_{1}, y_{2},\ldots$ be elements in $S$ such that $x<_{s}y_{j}$ for all $j$ and $x'\ll x$. We find $x''$ such that $x'\ll x''\ll x$. For each fixed $j$ there is some $n\in \mathbb{N}$ such that $(n+1)x''\ll n y_j$. Choosing a rapidly increasing sequence $(y_j^k)_k$ with supremum $y_j$, the sequence $(n\cdot y_j^k)_k$ increases rapidly and has  supremum $n\cdot y_j$. It follows that $x''<_s y_j^l\ll \infty$ for some $l=l(j)$. Set $y_j':=y_j^{l(j)}$ for each $j$. Now $x'\ll x'' <_sy_j' \ll\infty$ and our assumption implies that $x'\leq \sum_{j=1}^N y_j'\leq \sum_{j=1}^N y_j$.
\end{proof}

For equivalent notions of $\omega$-comparison we only consider simple $\Cu$-semigroups. The reason for restricting ourselves to simple $\Cu$-semigroups is that in the definition of the $\omega$-comparison, as it was introduced in \cite{OPR2}, there is no assumption on fullness of $x$. In this way, $\omega$-comparison is more similar to the strong CFP than the CFP. 
%Adding a suitable fullness condition in the definition of the $\omega$-comparison, one may pass to the non-simple setting with similar arguments as above.

\begin{proposition}\label{Omega}
Let $S$ be a simple $\Cu$-semigroup. Then the following statements are equivalent.
\begin{itemize}
\item[(i)] $S$ has $\omega$-comparison.
\item[(ii)] Whenever $(y_n)$ is a sequence of nonzero elements in $S_{\ll\infty}$ such that $y_n <_s y_{n+1}$ for all $n$, then $\sum_{n=1}^\infty y_n=\infty$ (in $S$). 
\item [(iii)] Whenever $(y_n)$ is a sequence in $S_{\ll\infty}$ such that $\lambda \left (\sum_{n=k}^\infty y_n\right )=\infty$ for all $k\in\mathbb{N}$ and all functionals $\lambda\in F(S)$, then $\sum_{n=1}^\infty y_n=\infty$.
\item [(iv)]Whenever $(y_n)$ is a sequence of nonzero elements in $S_{\ll\infty}$ such that $\lambda \left (\sum_{n=1}^\infty y_n\right )=\infty$ for all functionals $\lambda\in F(S)$, then $\sum_{n=1}^\infty y_n=\infty$.
\end{itemize}
\end{proposition}

\begin{proof}
By Lemma \ref{ReductionStepOmega} we may assume in the definiton of $\omega$-comparison that $y_j\ll\infty$ for all $j$. Now note, that instead of the condition that $x<_{s}y_{j}\ll \infty$ for all $j$, one may assume that $y_j<_s y_{j+1}\ll\infty$ for all $j$. Indeed, use simplicity to find for given $\sigma:=\sum_{j=k}^l y_j$ some $m\in\mathbb{N}$ such that $\sigma \leq m\cdot x$. Then, $$\sigma=\sum_{j=k}^l y_j\leq m\cdot x<_s\sum_{j=l+1}^{l+m}y_j.$$ Using this and starting with $\sigma:=y_1=\tilde{y}_1$, by iteration, we find the desired sequence $(\tilde{y}_j)_j$ with $\tilde{y}_j<_s \tilde{y}_{j+1}\ll\infty , j\in\mathbb{N}$, and $\sum_{j=1}^\infty y_j = \sum_{j=1}^\infty \tilde{y}_j$.  One now shows that (i) implies (ii) in an analogous way to the proof of (i) implies (ii) of Proposition \ref{PropCFP}, and the implication from (ii) to (i) is easy.

To connect the statements (i) and (ii) to statements (iii) and (iv) consider first the case that  there is $x\in S_{\ll\infty}$ with $\lambda(x)=\infty$ for all $\lambda\in F(S)$. Then, by Lemma \ref{InfiniteMultiples}, $\lambda=\lambda_\infty$ and every element is stably infinite. In this case, the conditions on the $y_n$ in (ii), (iii), and (iv) all just reduce to the condition that $y_n\neq 0$ for every $n$. Hence (ii) through (iv) are trivially equivalent in this case. We may therefore assume in what follows that there are functionals finite on $S_{\ll\infty}$.

To see that (iii) implies (ii), we use that $y_j<_s y_{j+1}$ implies that $\lambda(y_j)< \lambda(y_{j+1})$ for all functionals $\lambda$ finite on $S_{\ll\infty}$. The converse is shown by induction and by using Lemma \ref{UglyWayOut} as follows: If $y_j'=\sum_{k=a}^{b}y_j$ has been determined, then find $m$ as in Lemma \ref{UglyWayOut} for $y'_j$, and then find $c\in \mathbb{N}$ such that $m\cdot \lambda(y'_j)<\sum_{j=b+1}^c \lambda(y_j)$. Set $y'_{j+1}:=\sum_{j=b+1}^c y_j$.

Finally, to see that (iii) and (iv) are equivalent is easy as we have already reduced to the case that $\lambda(y_j)<\infty$ for all $j$.
\end{proof}

\begin{remark}
The assumption that all elements $y_j$ in {\rm(iv)} are nonzero is necessary (cf. Example \ref{Ex:Omega-noBeta}). 
\end{remark}

We will now define a new regularity property on comparison in an ordered abelian semigroup $W$ based on the value $\beta(x,y)$ for $x,y\in W$ as defined in Section \ref{SectionBeta}. Note that $\beta(x,y)=0$ if and only if for every $\epsilon>0$ there is $k\in \mathbb{N}$ such that $kx\leq \lfloor k\epsilon \rfloor y$, where $\lfloor a \rfloor$ denotes the largest smaller integer than $a$.

\begin{definition} ($\beta$-comparison)
Let $W$ be an ordered abelian semigroup.
We will say that $W$ satisfies the $\beta$-comparison property if whenever $x,y$ are two elements in $W$ with $\beta(x,y)=0$, then $x\leq y$.
\end{definition}

\begin{lemma}\label{BetaLemma}
Let $S$ be a simple $\Cu$-semigroup. Then the following statements are equivalent.
\begin{itemize}
\item[(i)] $S$ has $\beta$-comparison.
\item[(ii)] Whenever $\beta(x,y)=0$ for some non-zero $x$, then $y=\infty$.
\end{itemize}
\end{lemma}

\begin{proof}
Suppose $\beta(x,y)=0$ for some non-zero $x$, which implies $\beta(n\cdot x,y)=0$ for each $n\in\mathbb{N}$ by a simple computation.  Letting $x'\ll x$, we have $\beta(n\cdot x',y)=0$ by Lemma \ref{properties}.  Now (i) implies that $n\cdot x\leq y$ for all $n\in \mathbb{N}$, hence $y=\infty$ by simplicity. The other implication is trivial.
\end{proof}

The following lemma characterizes $\beta$-comparison using functionals and states.

\begin{lemma}\label{BetaAndFunctionals}
Let $S$ be a simple $\Cu$-semigroup. Then the following statements are equivalent for an element $y\in S$:
\begin{itemize}
\item[(i)] There is some nonzero $x\in S$ such that $\beta(x,y)=0$.
\item[(ii)] There is no faithful state $f\in \mathcal{S}(S,y)$.
\item[(iii)] $\lambda(y)=\infty$ for all functionals $\lambda\in F(S)$.
\end{itemize}
\end{lemma}

\begin{proof}
By Lemma \ref{FaithfulStates}, (ii) and (iii) are equivalent.

Suppose that the set of states $\mathcal{S}(S,y)$ normalized at $y$ is empty. Then, by Lemma \ref{NoStates}, there is $n\in \mathbb{N}$ such that $n\cdot y$ is properly infinite; hence, $n\cdot y=\infty$. In this case, $\beta(x,y)=0$ and $\lambda(y)=\infty$ for all functionals $\lambda$, so all three statements hold. 

Otherwise, there is a state $f \in \mathcal{S}(S,y)$ and $\beta(x,y)=\sup\{f(x)\ |\ f\in \mathcal{S}(S,y)\}.$ Thus, (i) and (ii) are also equivalent in this case. 
\end{proof}

With respect to Lemma \ref{BetaAndFunctionals}, one may wonder about the existence of a simple $\Cu$-semigroup $S$ with a {\bf stably finite} element $y\in S$ such that $\beta(x,y)=0$ for some nonzero $x$ (equivalently, $\lambda(y)=\infty$ for all $\lambda\in F(S)$). This question on existence has a positive answer, which is explained based on the C*-algebra constructed in \cite{Petzka13}:

\begin{example}\label{ExampleBetaVsQQ}
In \cite{Petzka13}, the second author constructed a stably finite projection $Q$ in the multiplier algebra of a separable stable simple C*-algebra $A$, which is of the form $Q=\sum_{j=1}^\infty {p_j}$, where the $p_j$ are pairwise orthogonal projections in $A$ and such that $\lambda(p_j)=\lambda(p_i)$ for all $i,j$ and all functionals $\lambda$. 
  
  Considering $a:=\sum_{j=1}^\infty \frac{1}{2^j} p_j$, we find a positive element $a\in A$ such that its Cuntz class, $y=\langle a \rangle$, satisfies that $\lambda(y)=\infty$ for all $\lambda$, hence $\beta(x,y)=0$. We will show that $y$ is stably finite. 
  
Indeed, that the multiplier projection $Q$ is stably finite is shown in \cite{Petzka13} by the existence of projections $g_n$ such that for each $n$ we have $g_n\npreceq n\cdot Q$. Assuming $n\cdot y=\infty$, we get in particular that $g_n\preceq n\cdot a$. By compactness of $\langle g_n\rangle$, there is some $N(n)\in\mathbb{N}$ such that $g_n\preceq n\cdot \sum_{j=1}^{N(n)} \frac{1}{2^j} p_j\sim \sum_{j=1}^{N(n)} p_j$. Now, the Cuntz subequivalence is just Murray-von Neumann subequivalence of projections; hence, $g_n\preceq \sum_{j=1}^{N(n)} p_j<Q$, a contradiction. It follows that $n\cdot y<\infty$ for all $n$ and $y$ is stably finite.
\end{example}

Our goal is now to relate $\beta$-comparison to $\omega$-comparison.

\begin{proposition} \label{BetaThenOmega}
 Let $S$ be $\Cu$-semigroup. If $S$ has $\beta$-comparison, then $S$ has $\omega$-comparison.
\end{proposition}

\begin{proof}
Assume that $S$ does not satisfy $\omega$-comparison. Then, there exists a sequence $\{y_{n}\}$, and $x,x'$ in $S$ such that $x'\ll x<_{s}y_{j}$ for all $j$ and $x'\not\leq y_{1}+y_{2}+\ldots y_{k}$ for any $k$. Let $y:=\sup_{n} (\sum^{n}_{j=1} y_{j})$ and notice that $x\not\leq y$. 

We have that $x\leq \infty\cdot  y_j \leq \infty \cdot y$ for all $j$, so $x\, \overline{\propto}\, y$. By Lemma \ref{properties},
$\beta(x',y)\leq \beta(x',\sum^{k}_{j=1} y_{j})$ for any $k$. Using Remark \ref{almost_zero}, one gets $\beta(x',y)\leq 1/k$ for all $k$.
Letting $k$ go to infinity, $\beta(x',y)=0$. Hence $\beta(x',y)=0$ for all $x'\ll x$, but $x\nleq y$, so $S$ does not satisfy $\beta$-comparison.
\end{proof}

The converse is not true for general (simple) $\Cu$-semigroups, but it is true for simple $\Cu$-semigroup satisfying the additional axioms (O5) and (O6) and containing no minimal element, which is the content of the following example and proposition.

\begin{example}\label{Ex:Omega-noBeta}
There exists a simple $\Cu$-semigroup $S$, such that $S$ has $\omega$-comparison, but no $\beta$-comparison.
\end{example}

\begin{proof}
Let $$S=\{0\}\cup \{1\}\cup\{\infty\}$$ with $1+1=\infty$ and $1$ compact. If one wants to exclude minimal (necessarily compact) elements in a $\Cu$-semigroup, then one can similarly consider $$\tilde{S}=\{0\}\cup (1,2]\cup\{\infty\}$$ with $x+y=\infty$ for any nonzero $x,y\in\tilde{S}$. Both semigroups are given the order inherited  from $\mathbb{R}$. Since the endpoint at 1 is not included in $\tilde{S}$, every element can be written as the supremum of a rapidly (i.e. strictly) increasing sequence. 

One checks that $S$ and $\tilde{S}$ are both simple satisfying the axioms (O1)--(O4). Moreover, note that $S$ satisfies (O5) and (O6). However, the order in $\tilde{S}$ is not almost algebraic, so (O5) fails in $\tilde{S}$.

We show that $S$ and $\tilde{S}$ have $\omega$-comparison, but no $\beta$-comparison. For the failure of $\beta$-comparison, note that $2x=\infty$ for every non-zero $x$ in $S$ and $\tilde{S}$. Thus, we have that $\beta(x,y)=0$ for arbitrary nonzero $x,y$. But in $S$ we have $1\neq \infty$, and in $\tilde{S}$ we have $3/2\neq\infty$. 

On the other hand, $S$ and $\tilde{S}$ both have $\omega$-comparison. To verify this, fix some nonzero $x$ in $S$ or $\tilde{S}$. Then for any sequence $(y_j)$ so that $x<_s y_j$ for all $j$, the $y_j$'s are necessarily non-zero, so $\sum_{j=1}^\infty y_j=\infty\geq x$.
\end{proof}

In the previous example, instead of the semigroup $S$, we could have considered more generally the semigroup $S_n=\{0,1,2,\ldots,n,\infty\}$ for some $n\in\mathbb{N}$, equipped with the natural order and the natural addition except that $x+y=\infty$ whenever the sum of $x$ and $y$ exceeds $n$ in $\mathbb{R}$.  These $\Cu$-semigroups were also studied in \cite{APT14}. The following proposition states that, in the class of simple $\Cu$-semigroups satisfying (O1)--(O6), these are the only semigroups that distinguish $\omega$-comparison from $\beta$-comparison.

\begin{proposition}\label{BetaVsOmega}
 Let $S$ be a simple $\Cu$-semigroup satisfying axioms (O1)--(O6) and $S\neq S_n$ for any $n\in\mathbb{N}$. If $S$ has $\omega$-comparison, then $S$ has $\beta$-comparison (and hence  they are equivalent by Proposition \ref{BetaThenOmega}).
\end{proposition}

\begin{proof}
Let $y\in S$ such that $\beta(x,y)=0$ for some nonzero $x$. Then, by Lemma \ref{BetaAndFunctionals} we have that $\lambda(y)=\infty$ for all functionals $\lambda\in F(S)$. Pick a rapidly increasing sequence $(y_i)$ of nonzero elements with supremum $y$. 

As $\lambda(y)=\infty$ and $\lambda$ preserves suprema, we can (after possibly changing to a subsequence) assume that $\lambda(y_i)+i\leq \lambda(y_{i+1})$ for all $i$ and all $\lambda$. (Here we are using compactness of $F(S)$, which was shown in \cite[Theorem 4.8]{ERS11}). Using (O5), we can find for each $i$ some $z_i\in S$ such that $y_i+z_i\leq y_{i+2}\leq y_{i+1}+z_i$. 

We will distinguish between the case that $\lambda(y_i)=\infty$ for all $\lambda\in F(S)$ (for some and hence all (nonzero) $y_i$, c.f. Lemma \ref{InfiniteMultiples}) and the case that there is some functional $\mu$ with $\mu(y_i)<\infty$ for all $i$. 

First suppose the latter. Then $\mu(z_i)\geq \mu(y_{i+2})-\mu(y_{i+1})\geq i+1$ for such a functional finite on $S_{\ll\infty}$. It follows in particular that $z_i$ is nonzero. If now $\nu$ is a functional, then either $\nu$ is infinite on $S_{\ll\infty}$ and hence on all of $S\setminus\{0\}$, or $\nu$ is finite on each $y_j$ and $\nu(z_i)\geq i+1$. In particular, we get for all functionals $\lambda$ that $\lambda \left (\sum_{i=k}^\infty z_{2i+1}\right )=\infty$ for all $k$. Now, $\omega$-comparison implies that $\sum_{i=1}^\infty z_{2i+1}=\infty$ (by Proposition \ref{Omega}). But, for any $N\in \mathbb{N}$, we have $$\sum_{i=1}^{2N+1} z_{2i+1}\leq y_1+z_1+z_3+z_5+\ldots+z_{2N+1} \leq y_3+z_3+z_5+\ldots+z_{2N+1}\leq\ldots\leq y_{2N+1}.$$
Hence $$\infty=\sum_{i=1}^\infty z_{2i+1}\leq \sup_i y_i=y.$$

In the other case, in which the only functional on $S$ is the trivial functional $\lambda_\infty$ assigning $\infty$ to all nonzero $z\in S$, we only need to find for each $y\in S$ a sequence of nonzero $z_j$'s such that $\sum_{j=1}^\infty z_j\leq y$. Then, an application of $\omega$-comparison on the sequence $(z_j)_j$ yields the desired conclusion that $y=\infty$. That such a sequence can be found in the simple $\Cu$-semigroup $S\neq S_n$ satisfying axioms (O1)--(O6) follows from the fact that these semigroups have the Glimm halving property (\cite{Robert13}) (i.e., for every nonzero $x\in S$ there is some nonzero $z\in S$ such that $2z\leq x$).
\end{proof}

In \cite[Proposition 2.17]{OPR2} it is proved that if a complete abelian ordered semigroup satisfies $\omega$-comparison, then it also satisfies the corona factorization property. In the same paper it was left unanswered whether the converse holds. The example below shows that both properties are not equivalent for $\Cu$-semigroups, and hence neither for general complete abelian ordered semigroups. It remains an open question whether the two notions are equivalent for $\Cu$-semigroups $S$ coming from a C*-algebra, i.e., for $S=\Cu(A)$ for some C*-algebra $A$.

\begin{example}\label{Ex:CFP-noW}
There exists a simple $\Cu$-semigroup that satisfies the corona factorization property but does not have the $\omega$-comparison property.
\end{example}
\begin{proof}
Let $$S=[0,1]\cup\{\infty\}$$ equipped with the usual order 
and addition, except that if $x,y\in S$ are such that $x+y> 1$ in $\mathbb{R}$, we set $x+y=\infty$. 

It is easy to check that $S$ is simple, totally ordered and that it satisfies the (strong) corona factorization property.

On the other hand, let us check that $S$ does not satisfy $\omega$-comparison. To do so, firstly note that, for $0\leq x,y\leq 1$ in $S$, $x\ll y$ if and only if $x<y$, and that $\infty\ll\infty$. Now, consider the sequence $\{y_{n}\}=\{1/2^{n+1}\}$ and the elements $x=1\,, x'=3/4$. Clearly $x'\ll x$ and we get $x<_{s} y_{j}$ for all $j$, since in fact $x<_sy$ holds for arbitrary $x,y$ in $S$.

But as $\sum^{\infty}_{j=1}y_{j}=1/2 \ngeq  3/4$, we conclude
that $S$ does not satisfy the $\omega$-comparison property. 
\end{proof}

%\textcolor{red}{Here?} 
%
%\begin{remark}
%It seems predictable that if one has a simple ordered semigroup with Corona Factorization property and refinement, then one has $\omega$-comparison property. Nevertheless, we don't %believe it is important to further study due there is no natural condition on C*-algebras theory that induce refinement to the Cuntz semigroup.
%
%Note that the above example is not a refinement semigroup. Indeed, let $W=[0,1]\cup\{\infty\}$ as before. If one considers $a_{i}=2/3$ and $b_{i}=5/6$ for $i=1,2$ one has 
%$a_{1}+a_{2}=b_{1}+b_{2}$. Then if $W$ was a refinement semigroup, there would exist $z_{i,j}$ in $W$ such that 
%$2/3=z_{1i}+z_{2i}$, and $5/6=z_{i1}+z_{i2}$ for $i=1,2$. Using the refinement property either $z_{11}\geq 1/3$ or $z_{21}\geq 1/3$. Assume $z_{11}\geq 1/3$ and $z_{21}\leq 1/3$. 
%Then $z_{22}\geq 3/6$, which implies that $z_{12}\leq 1/6$. It subsequently implies that $z_{11}=2/3$ and $z_{21}=0$. This is a contradiction.
%\end{remark}

There is also a stably finite $\Cu$-semigroup distuingishing $\omega$-comparison and the CFP.

\begin{example}\label{Ex:stablyNoFunctionals}
 Let $$S=\{ (0,0)\} \cup ((0,1]\cup\{\infty\})\times (0,\infty] ,$$
with addition defined by componentwise addition and with the additional condition that if in the first component we have $x+y>1$ , then $x+y=\infty$. Namely, $(x,r)+(y,s)=(\infty,r+s)$ whenever $x+y>1$ in $\mathbb{R}$. The order is componentwise with the natural ordering in each component. Note that the relation of compact containment is given by componentwise strict inequalities and with $(\infty,r)\ll(\infty,s)$ whenever $r< s$.
 
It is easy to check that $S$ is simple and satisfies all axioms {\rm (O1)--(O6)}. Note that all elements except those of the form $(x,\infty)$ are stably finite. This makes $S$ a simple stably finite $\Cu$-semigroup. 

If $\lambda$ is a nonzero functional on $S$, then $\lambda((x,\infty))=\infty$ for all $\lambda\in F(S)$. Indeed, by construction, there exists $m\in\mathbb{N}$ such that $m(x,\infty)=\infty$, so $m\cdot \lambda((x,\infty))= \lambda((\infty,\infty))=\infty$.  Therefore, by Proposition \ref{BetaAndFunctionals}, there is some nonzero $(x,r)\in S$ such that $\beta((x,r),(y,\infty))=0$, while $(y,\infty)$ is stably finite. Hence, $S$ does not have $\beta$-comparison, and neither does $S$ satisfy $\omega$-comparison by Proposition \ref{BetaVsOmega}.
On the other hand, it is easy to see with Proposition \ref{PropCFP} that $S$ satisfies the (strong) corona factorization property.
 \end{example}

\begin{remark}
One can even find an algebraic, simple, stably finite $\Cu$-semigroup with the CFP and failing $\omega$-comparison by a small modification of the previous example. Take 
$$S=\{ (0,0)\} \cup (((\mathbb{Q}\cap(0,1])\sqcup (0,1])\cup\{\infty\})\times ((\mathbb{Q}\cap(0,\infty))\sqcup (0,\infty])\,$$
with addition and order similar to before, but now addition and order in each component are defined as in the Cuntz semigroup of the universal UHF-algebra (see e.g. \cite{BPT}). By construction, the semigroup $S$ is algebraic, and one shows the required properties of $S$ in the same way as above.
\end{remark}

Note that the simple $\Cu$-semigroup in Example \ref{Ex:CFP-noW} is neither stably finite nor purely infinite. This behavior is in general ruled out by the property of $\beta$-comparison.

\begin{proposition}\label{betaImpliesDichotomy}
Let $S$ be a simple $\Cu$-semigroup with $\beta$-comparison. Then $S$ is either stably finite or purely infinite.
\end{proposition}

\begin{proof}
Suppose $S$ is not stably finite, so that there is some $x\ll\infty$ which is not finite. In other words, $\infty$ is compact, and for every nonzero $y\in S$ some finite multiple of it is properly infinite.  By Lemma \ref{prop1}, $\beta(x,y)=0$ for every nonzero $x,y$. By $\beta$-comparison, every nonzero element is infinite and $S=\{0,\infty\}$, i.e., $S$ is purely infinite.
\end{proof}

Hence, in the simple case, $\beta$-comparison implies the dichotomy of being either stably finite or purely infinite. By Proposition \ref{BetaVsOmega}, $\omega$-comparison implies the same dichotomy of a simple $\Cu$-semigroup $S$ satisfying (O1)--(O6) and different from $S_n$ for any $n$.

One sees from the above example that the CFP allows for the existence of both finite and infinite elements in simple $\Cu$-semigroups. In particular, the CFP is not equivalent to the following stronger statement:

$$\mbox{If }m\cdot \sum_{n=1}^\infty y_n=\infty\mbox{, then }\sum_{n=1}^\infty y_n=\infty.$$

We now turn our attention to $\Cu$-semigroups without the CFP (therefore without $\omega$-comparison and $\beta$-comparison) that are neither stably finite nor purely infinite. The next result provides a  characterization of a simple $\Cu$-semigroup not having the CFP, which we subsequently use to find an explicit simple $\Cu$-semigroup without the CFP that is neither stably finite nor purely infinite. However, this semigroup does not satisfy the axiom (O6). (See Theorem \ref{RordamsAlgebraNoCFP} for the existence of a simple $\Cu$-semigroup with (O6) and without the CFP that is neither stably finite nor purely infinite, which is given as the Cuntz semigroup of a C*-algebra. This Cuntz semigroup, however, has not been computed yet.)

\begin{proposition}\label{CharactCFP}
Let $S$ be a simple $\Cu$-semigroup, containing a finite compact element, and with $\infty\ll\infty$.  Then $S$ does not have the CFP if and only if there is a sequence of elements $(z_n)$ in $S_{\ll\infty}$ such that $2z_n=\infty$ and $\sum_{n=1}^\infty z_n<\infty$.
\end{proposition}

\begin{proof}
It is clear that the existence of such a sequence implies the lack of CFP. Conversely, suppose $S$ does not have the CFP. Then, by Proposition \ref{PropCFP}, there is $m\in \mathbb{N}$ and a sequence $(y_n)$ in $S_{\ll\infty}$ such that $m\cdot \sum_{n=k}^\infty y_n=\infty$ for all $k$; however, $\sum_{n=1}^\infty y_n\neq \infty$. Choosing $m$ to be minimal, replacing each $y_n$ with a suitable multiple of itself and possibly discarding a finite number of $y_n$'s, we may assume that $m=2$. Now $2\cdot \sum_{n=k}^\infty y_n=\infty$ for all $k$ and $\infty$ is compact. Thus, for each $k$ there is some $N(k)$ such that $2\cdot \sum_{n=k}^{N(k)} y_n=\infty$. Choose $z_1=\sum_{n=1}^{N(1)}y_n$, and then choose $z_{n+1}$ inductively from $z_n=\sum_{n=s}^{t}y_n$ to be $z_{n+1}=\sum_{t+1}^{N(t+1)}y_n$. Then, the sequence $(z_n)$ in $S_{\ll\infty}$ satisfies $2z_n=\infty$, and $\sum_{n=1}^\infty z_n =\sum_{n=1}^\infty y_n<\infty$ as required.
\end{proof}

Explicitly, we have the following example. 
\begin{example}\label{Ex:nonStablynoPurely}
We construct an example of a simple $\Cu$-semigroup satisfying (O5), without the CFP, and which is neither stably finite nor purely infinite. However, our example does not satisfy the axiom (O6) of almost Riesz refinement.
\end{example}

\begin{proof}
Let $$S=[0,1]^\mathbb{N}\cup\{\infty\},$$
with addition given by componentwise addition and with the relation that $x+y=\infty$, whenever any component exceeds 1. One checks that $S$ is simple and that it satisfies all the required axioms, i.e. (O1)--(O5). %Note that here the compact containment relationship is given by strict componentwise order, but just by the sequences that have a finite amount of non-zero components.

Letting $y_n$ denote the element in $S$, which is 1 at position $n$ and zero elsewhere, Proposition \ref{CharactCFP} applies to show that $S$ does not have the CFP.

To see that (O6) does not hold, consider $y_1=(1,0,0,\ldots)$ and $y_2=(0,1,0,0,\ldots)$. We have $y_1\leq y_2+y_2=\infty$, but we cannot find any nonzero elements $x_1,x_2\leq y_1,y_2$.
\end{proof}

Alternatively, in the previous example one could have used $S=\{0,1\}^\mathbb{N}\cup\{\infty\}$ instead, but our aim was to show that one can guarantee the non-existence of minimal nonzero elements in $S$. The failure of axiom (O6) in the last example can be generalized as follows.

\begin{proposition}(\cite[Lemma 5.1.18]{APT14})
Let $S$ be a simple $\Cu$-semigroup satisfying (O6). Then for any finite number of elements $y_1,\ldots, y_n$ in $S_{\ll\infty}$ there is some nonzero $z\in S$ such that $z\leq y_j$ for all $j$.
\end{proposition}

%\begin{proof}
%We consider first the case $n=2$. As $y_1,y_2\ll\infty$, there is some $k$ with $y_1\leq k\cdot y_2$. If $k=2$, then we choose some nonzero $y_1'\ll y_1$ and apply (O6) to the %inequality $y'_1\ll y_1\leq y_2+y_2$ to find $z_1\leq y_1,y_2$ and $z_2\leq y_1,y_2$ such that $z_1+z_2\geq y_1'$. In particular, either $z_1$ or $z_2$ is nonzero and below both $y_1$ %and $y_2$. Else, if $k>2$, then apply (O6) to the inequality $y_1\leq y_2+(k-1)\cdot y_2$ to find $z_1\leq y_1,y_2$ and $z_2\leq y_1,(k-1)\cdot y_2$ such that $z_1+z_2\geq y_1'$. %Again, either $z_1$ or $z_2$ is nonzero. If $z_1\neq 0$, then we are done with $z=z_1$. Otherwise apply (O6) to the inequality $0\neq z_2'\ll z_2\leq y_2+(k-2)y_2$ to find $z_3\leq %y_2,z_2$ (hence $z_3\leq y_2,y_1$) and $z_4\leq (k-2)y_2,z_2$ (hence $z_4\leq (k-2)y_2, y_1$) with either $z_3$ or $z_4\neq 0$. Iterate, if necessary, until $0\neq z_{2s}\leq y_1, (k-%s)y_2=2y_2$ and apply the arguments from before to find $0\neq z\leq y_1,y_2$.
%
%Having verified the case $n=2$, one proves the general case by induction.  
%\end{proof}

\begin{remark}
Theorem \ref{RordamsAlgebraNoCFP} shows the existence of a simple C*-algebra $A$ such that its Cuntz semigroup $\Cu(A)$ is neither stably finite nor purely infinite and fails to have the CFP. On the other hand, it seems difficult to write down an explicit example of a simple $\Cu$-semigroup $S$, neither stably finite nor purely infinite, satisfying all the axioms (O1)--(O6) and failing to satisfy the CFP. By the previous proposition, in such a semigroup, for any finite number of elements $y_1,\ldots,y_n$ one can find a nonzero element $z\in S$ such that $z\leq y_j$ for all $j=1,2,\ldots,n$. However, if the CFP fails in $S$ and this failure is witnessed by a sequence $(y_j)$, then there is no nonzero $z\in S$ such that $z\leq y_j$ for all $j\in\mathbb{N}$, as otherwise $\sum_{j=1}^\infty y_j\geq \infty\cdot z=\infty$.

%Notice, however, that this characterization is so useful in the case of $A=\otimes^\infty_{n=1}C(S^2)$. Indeed, considering the sequence $\{x_n\}$ in $\Cu(A)$ defined by $x_n=[1\otimes^{n-2}...1\otimes p\otimes1\otimes 1\ldots]$, where $p$ denotes the Bott projection, one has that $[1\otimes 1\otimes\ldots]\leq 2x_n$ for all $n$. Hence, defining $z_n=\sum^{\infty}_{k=1}x_{n_k}$, where $\{n_k\}^\infty_{k=1}$ just denotes infinite partitions of the natural numbers of infinite size, one gets the desired sequence since the element $[1\otimes 1\otimes\ldots]$ is full.
\end{remark} 

The CFP is closely related to the property {\rm (QQ)}, which was introduced in \cite{OPR2}.

\begin{definition}(\cite{OPR2})
A positively ordered abelian semigroup $W$ satisfies the property {\rm (QQ)} if every element in $W$, for which a multiple is properly infinite, is itself properly infinite.  
\end{definition}

The following relations are immediate.

\begin{proposition}\label{PropQQ}
Let $S$ be a simple $\Cu$-semigroup. 
\begin{itemize}
\item[(i)] If $S$ has $\beta$-comparison, then $S$ has property {\rm (QQ)}. 
\item[(ii)] If $S$ has {\rm (QQ)}, then $S$ has the CFP.
\end{itemize}
\end{proposition}

\begin{proof}
For (i), let $x\in S$ and $n\in \mathbb{N}$ with $n\cdot x=\infty$. Then $\beta(z,x)=0$ for all $z$, and by $\beta$-comparison, $x=\infty$. Statement (ii) is trivial (with Proposition \ref{PropCFP}).  
\end{proof}

\begin{remark}
Attempting to prove the converse to {\rm (i)} in the most direct fashion, one would hope that $\beta(x,y)=0$ for some nonzero $x$ (equivalently, $\lambda(y)=\infty$ for all functionals $\lambda$) implies that some multiple of $y$ should be infinite. That this does not hold in general has already been noted in Example \ref{ExampleBetaVsQQ}. Hence, to show that the converse to (i) holds, one would need that the existence of some $y\neq \infty$ with $\beta(x,y)=0$ for some nonzero $x$ implies the existence of some $z\neq \infty$ in $S$ (possibly $z\neq y$) such that $n\cdot z=\infty$. 
\end{remark}

We introduce a new property related to the existence of both finite and infinite elements in $S_{\ll\infty}$:

\begin{definition}
 A complete abelian positively ordered semigroup $W$, containing a largest element $\infty$, has cancellation of small elements at infinity, if whenever $x$ and $y$ are elements in $W$ with $x\ll \infty$, $y\neq 0$ and $x+y=\infty$, then $y=\infty$.
\end{definition}

It is clear that if $S\neq \{0,\infty\}$ and $\infty$ is compact in $S$, then cancellation of small elements at infinity fails. Example \ref{Ex:stablyNoFunctionals} shows that cancellation of small elements at infinity can also fail in a stably finite $\Cu$-semigroup satisfying all axioms  (O1)--(O6). It is not known (but possibly expected) whether cancellation at infinity holds for the Cuntz semigroup $\Cu(A)$ of a simple stably finite C*-algebra $A$.

The next result shows that cancellation of small elements at infinity holds when $S$ satisfies either {\rm (QQ)} or $\omega$-comparison or $\beta$-comparison.

\begin{proposition}\label{QQImpliesCAI}
Let $S$ be a simple $\Cu$-semigroup.
\begin{enumerate}[\rm(i)]
 \item  If $S$ has property {\rm (QQ)}, then it has cancellation of small elements at infinity.
 \item If $S$ has $\beta$-comparison, then it has cancellation of small elements at infinity.
\end{enumerate}
\end{proposition}

\begin{proof}
Suppose $x\ll \infty$ and $x+y=\infty$. Since $x\ll \infty$ and also using simplicity, there is $n\in \mathbb{N}$ such that $x\leq n\cdot y$. Hence $(n+1)\cdot y=\infty$. By property (QQ), we get $y=\infty$. The second statement easily follows from combining {\rm (i)} with Proposition \ref{PropQQ}.
% We will prove the second fact in a contrapositive way. Let $x\ll\infty$ and $y$ in $S$ such that $x+y=\infty$ and $y\neq\infty$. As before, find $n\in\mathbb{N}$ such that $x\leq ny$; hence, $(n+1)y=\infty$. It implies that $\lambda(y)=\infty$ for all $\lambda\in F(S)$, i.e. $\beta(x,y)=0$ by Proposition \ref{BetaAndFunctionals}. If $S$ satisfied $\beta$-comparison, it would imply $x\leq y$. The same argument would be true for $nx$ for all $n\in\mathbb{N}$; hence, writing $\infty=\sup(nx)$, one would have $y=\infty$, the contradiction.
\end{proof}
% 
% Letting $S$ be a simple $\Cu$-semigroup, note that the above results induce the next diagram:
% $$\xymatrix{&\beta-comparison \ar@{=>}[rd]  \ar@{=>}[ld] &  \\ 
% {\rm (QQ)} \ar@{=>}[rr]&    & \mbox{cancellation of small elements at }\infty,} $$
% which can not be commutative due to $\beta$-comparison and {\rm (QQ)} imply dichotomy and cancellation of small elements not. Indeed, Remark \ref{ExStablyfinite} shows the ex. 

As we shall see below, the converse to Proposition \ref{QQImpliesCAI}(i) holds for certain simple $\Cu$-semigroups. Recall that a Cuntz semigroup is called algebraic, if every element can be written as the supremum of an increasing sequence of compact elements.

%\begin{lemma}\label{SimplifiedCFP}
%Let $S$ be a simple $\Cu$-semigroup with cancellation of small elements at infinity. Then the CFP is equivalent to the following property:
%\begin{itemize}
%\item Given a sequence $(y_n)$ in $S_{\ll\infty}$ such that $m\cdot \sum_{n=1}^\infty y_n=\infty$ for some $m$, then $\sum_{n=1}^\infty y_n=\infty$.
%\end{itemize}
%\end{lemma}
%
%\begin{proof}
%We use version (iv) from Proposition \ref{PropCFP} as the suitable rephrasing of the CFP. 
%
%If $\infty$ is compact then, by cancellation of small elements at infinity and by simplicity, $S=\{0,\infty\}$ and the statement holds. If, on the other hand, $\infty$ is not compact, then %$y_n$ is nonzero for infinitely many $n$. In this case, cancellation of small elements at infinity shows that  $m\cdot \sum_{n=1}^\infty y_n=\infty$ holds if and only if $m\cdot %\sum_{n=k}^\infty y_n=\infty$ for every $k$.
%\end{proof}
%
%

\begin{proposition}\label{CFPVsQQ}
Let $S$ be a simple algebraic $\Cu$-semigroup with {\rm (O5)}.  Then $S$ has property {\rm (QQ)} if and only if $S$ has both the CFP and cancellation of small elements at infinity.  
\end{proposition}

\begin{proof}
Proposition \ref{QQImpliesCAI} and Proposition \ref{PropQQ} show the 'only if'-direction. 

Using Proposition \ref{PropCFP} one sees that, under the assumption of cancellation of small elements at infinity, the CFP can be rephrased as the statement that if $(y_j)_j$ is a sequence in $S_{\ll\infty}$ such that $m\cdot \sum_{j=1}^\infty y_j=\infty$ for some $m\in \mathbb{N}$, then $\sum_{j=1}^\infty y_j=\infty$. Loosely speaking, the CFP equals property {\rm (QQ)} for elements of the form $y=\sum_{j=1}^\infty y_j$ with $y_j\ll\infty$ for all $j$. Using (O5) and the assumption that $S$ is algebraic, one see that every element in $S$ can be written as $\sum_{j=1}^\infty y_j$ for suitable $y_j\ll\infty$. Hence, property {\rm (QQ)} holds.
\end{proof}

Proposition \ref{CFPVsQQ} can be generalized to the simple non-algebraic case with a minor technical limitation. (The `only if'-direction holds for a general simple $\Cu$-semigroup.)

\begin{proposition}\label{CFPVsQQPart2}
Let $S$ be a simple algebraic $\Cu$-semigroup with (O5) and with cancellation of small elements at infinity. Suppose that $S$ does not have property {\rm (QQ)} and the failure of {\rm (QQ)} is witnessed by an element $x\in S$ and some $m>2$, such that $m\cdot x=\infty$, but $(m-1)x<\infty$. Then $S$ does not have the CFP.  
\end{proposition}

We omit the proof as the arguments are similar to the ones in (the first part of) the proof of Theorem \ref{CFPVsQQAlg}, in which we overcome the technical limitation (of needing $m$ to be strictly greater than 2) and prove the conclusion of Proposition \ref{CFPVsQQ} for any $\Cu$-semigroup $S=\Cu(A)$ coming from a simple C*-algebra $A$. 

%%% THIS IS HOW THIS WORKS:
% Given y, find a rapidly increasing sequence with supremum y, such that $y_j\neq y_k$. Use (O5) to find for each $y_j\lly_{j+1}\lly_{j+2}$ some $z_j$ such that y_j+z_j\leq y_{j+2}\leq y_{j+1}+z_j$. Set $z_0:=y_1$. Then  the sum of all even $z_2k$ is smaller than $y$, the sum of all odd $z_{2k+1}$ is smaller than $y$. so the sum of all $z_j$'s is smaller than $2y$, hence finite. But ths sum of all $z_j$ is bigger than $y$, so some multiple of it is infinite and $(z_n)_n$ is the sequence we were looking for.

%%%%%%%%%%%%%%%%%%%%%%%%%%%%%%%%%%%%%%%%%%%%%%%%%%%%%%%%%%%%%%%%%%%%%%%%%%%%%%%%%%%%%%%%%%%%%%%%%%%%%%%%%%%%%%%%%%%%%%%%%%%%%%%%%%%%%%%%%%%%%%%%%%%%%%%%%%%%%%%%%%%%%%%%%%%555
%%%%%%%%%%%%%%%%%%%%%%%%%%%%%%%%%%%%%%%%%%%%%%%%%%%%%%%%%%%%%%%%%%%%%%%%%%%%%%%%%%%%%%%%%%%%5

\section{Applications to the Cuntz semigroup of a C*-algebra}\label{SectionApplications}

In this final section we use the results obtained in the previous sections at the level of general $\Cu$-semigroups to $\Cu$-semigroups arising from C*-algebras, i.e., the case where $S=\Cu(A)$. Theorem \ref{CFPVsQQAlg} shows that, for any simple C*-algebra $A$, the CFP in combination with cancellation of small elements at infinity is equivalent to property (QQ). We summarize the relations between all regularity properties studied in this paper in Theorem \ref{thm:end}. Finally, we show in Theorem \ref{RordamsAlgebraNoCFP} that the C*-algebra described in \cite{R03}, containing both a non-zero finite projection and an infinite projection, does not have the CFP. 

\begin{definition} Let $W$ be an ordered abelian semigroup. We
  say that $W$ has the weak halving property if for every
$x \in W$ there are $y_1,y_2 \in W$ such
  that $y_1+y_2 \le x$ and $x \propto y_j$ for $j=1,2$.
\end{definition}

Note that if an ordered abelian semigroup $W$ has the weak halving property, then inductively one can find a sequence $(y_n)_n$ of elements in $W$ such that for each $n\in \mathbb{N}$ one has $y_1+y_2+ \cdots+y_n\le x$ and $x \propto y_n$ for all $n$. In the case of a complete ordered semigroup, we also get $\sum_{j=1}^\infty y_j\leq x$.

\begin{lemma} \label{lm:halving}
  Let $A$ be a unital simple \Cs{} not of type I. Denoting by $W(A)$ the (original) Cuntz semigroup given by equivalence classes of positive elements in matrix algebras over $A$, it follows that
  $W(A)$ has the weak halving property.
\end{lemma}

\begin{proof}Let $x \in W(A)$ be given. Upon replacing $A$ by a
  matrix algebra over $A$, we may assume that $x = \langle a \rangle$
  for some positive element $a$ in $A$. We may also assume that $a$ is
  non-zero (as it is trivial to halve the zero-element). Take a
  maximal abelian sub-\Cs{} $D$ of $\overline{aAa}$. Then $D$ is
  infinite dimensional (by the assumption that $A$ is not of type I),
  and hence contains two non-zero pairwise orthogonal positive
  elements $b_1,b_2$. Put $y_j = \langle b_j \rangle$. Then $y_1+y_2 =
  \langle b_1 +b_2 \rangle \le \langle a \rangle = x$, and $x \propto
  y_j$ for $j=1,2$, because $W(A)$ is algebraically simple, i.e. $x\propto y$ for all $x,y\in W(A)$.
\end{proof}

The next example shows that the $\Cu$-semigroups $S_n$ (see the paragraph before Proposition \ref{BetaVsOmega}) can not arise as the Cuntz semigroup of a C*-algebra (cf. \cite[Remark 5.1.17]{APT14}). Hence, this shows that $\omega$-comparison and $\beta$-comparison are equivalent properties for any $\Cu$-semigroup coming from a C*-algebra.

\begin{example}\label{elementaryNon} Let $n$ be a natural number and let 
$$S_n = \{0,1,2, \dots,n,\infty\}$$
as in Proposition \ref{BetaVsOmega}. $S_n$ is simple and satisfies $\omega$-comparison, but not the property (QQ) (the element $1$ is finite but $\infty = (n+1)
\cdot 1$ is properly infinite), therefore neither $\beta$-comparison.

However, note that the semigroup $S_n$  fails to have the weak halving property, hence it
can not be the Cuntz semigroup of a simple C*-algebra by Lemma \ref{lm:halving}. (Note that if $\Cu(A)=S_n$  for the completed Cuntz semigroup of a simple C*-algebra $A$, then also $W(A)=S_n$.)
\end{example}

\begin{remark}
Leonel Robert shows in \cite{Robert13} that a simple Cuntz semigroup $S$ with axioms {\rm (O1)--(O6)} has either Glimm halving (for every nonzero $x\in S$ there is some nonzero $z\in S$ such that $2z\leq x$) or $S=S_n$ for some $n\in\mathbb{N}\cup\{\infty\}$. By the C*-algebraic proof of the weak halving property above, we can rule out the possibility of $S=S_n$ for some $n\in\mathbb{N}$. It follows that every Cuntz semigroup $S=\Cu(A)$, coming from a simple nonelementary C*-algebra $A$, has the Glimm halving property.
\end{remark}

We characterize $\omega$-comparison for simple $\Cu$-semigroups $S=\Cu(A)$ coming from a C*-algebra.

\begin{proposition}(cf. \cite{BRTTW})\label{BRTTW}
If there is a simple C*-algebra $A$ such that $S=\Cu(A)$, then the $\omega$-comparison is also equivalent to the following statements (see Propostion \ref{Omega}).
\begin{itemize}
\item[(iv)] If $y\in S$ is such that $\lambda(y)=\infty$ for all functionals $\lambda \in F(S)$, then $y=\infty$.
\item[(v)] $A$ is regular, i.e., whenever $D$ is a non-unital hereditary subalgebra of $A\otimes\mathcal{K}$ with no bounded quasitrace, then $D$ is stable.
\end{itemize}
\end{proposition}

\begin{proof}
By Lemma \ref{lm:halving} (and Example \ref{elementaryNon}), $S\neq S_n$ for any $n\in \mathbb{N}$. It follows from Proposition \ref{BetaVsOmega} that $S$ has the $\omega$-comparison if and only if it has the $\beta$-comparison. This shows the equivalence of $\omega$-comparison and {\rm (iv)} with the help of Lemma \ref{BetaLemma} and Lemma \ref{BetaAndFunctionals}.

Since {\rm (iv)} and {\rm (v)} both imply dichotomy (by Proposition \ref{betaImpliesDichotomy} and \cite[Lemma 4.6]{NgCFP} respectively), it suffices to show the equivalence of {\rm (iv)} and {\rm (v)} in the case that all projections in the stabilization of $A$ are finite. In this case, the desired equivalence was shown in \cite[Theorem 4.2.1 (i) \& (iii)]{BRTTW} (see also the last paragraph of \cite[Section 3]{BRTTW}). 
\end{proof}

\begin{remark}\label{Rm:BRTTW}
Notice that it follows from Proposition \ref{BRTTW} that for a simple C*-algebra $r_{A,\infty}$ (radius of comparison with respect to $\infty$ (see \cite{BRTTW} for further details)) is zero if and only if $\Cu(A)$ satisfies $\omega$-comparison. Combining this with Proposition \ref{betaImpliesDichotomy}, ones gets that a simple C*-algebra $A$ with $r_{A,\infty}=0$ is either stably finite or purely infinite.
\end{remark}

It was shown in \cite{OPR2} that a $\sigma$-unital C*-algebra $A$ has the corona factorization property (i.e., every full projection in $\mathcal{M}(A\otimes \mathcal{K})$ is properly infinite) if and only if $\Cu(A)$ has the CFP. We discussed in Section 4 that the corona factorization property might allow for the existence of both finite and infinite compact elements in a simple $\Cu$-semigroup (Example \ref{Ex:CFP-noW}). One may therefore ask the question whether the simple nuclear C*-algebra $A$ containing both a non-zero finite and an infinite projection constructed in \cite{R03} has the CFP. That this is not the case is proven in Theorem \ref{RordamsAlgebraNoCFP}. Before proving Theorem \ref{RordamsAlgebraNoCFP}, let us first state the following result, which follows immediately from Proposition \ref{CharactCFP}. (But note that the proof to Theorem \ref{RordamsAlgebraNoCFP} only requires the trivial direction of Propostion \ref{NoCFPCharact}.)

\begin{proposition}\label{NoCFPCharact}
Let $A$ be a simple C*-algebra containing both a finite and an infinite projection. Then $A$ does not have the CFP if and only if there is a sequence of elements $(z_n)_n$ in $\Cu(A)$ such that $z_n\ll\infty$, $2z_n=\infty$ and $\sum_{n=1}^\infty z_n<\infty$.
\end{proposition}

\begin{theorem}\label{RordamsAlgebraNoCFP}
The Cuntz semigroup of the simple nuclear C*-algebra $C$ constructed in \cite{R03}, containing both a non-zero finite and an infinite projection, does not have the CFP.
\end{theorem}

\begin{proof}
We will remind the reader of some key features of the construction retaining the notation from \cite{R03}. The algebra in question is a crossed product $C=D\rtimes_\alpha \mathbb{Z}$. We will then find the desired elements for the application of Proposition \ref{NoCFPCharact} right from its construction. 

At first, let $A:=C(\prod_{j=1}^\infty S^2,\mathcal{K})$. There is an injective map $\varphi$ from $A$ into its multiplier algebra $\mathcal{M}(A)$ with certain properties (see \cite[Proposition 5.2]{R03}), which extends to an injective map $\bar{\varphi}:\mathcal{M}(A)\rightarrow \mathcal{M}(A)$. This extension $\bar{\varphi}$ induces an inductive sequence with limit $B$ given by
$$\xymatrix{\mathcal{M}(A) \ar[r]^{\bar{\varphi}} \ar@/_1.5pc/[rrrr]_{\mu_{\infty,0}}  & \mathcal{M}(A) \ar[r]^{\bar{\varphi}}&\mathcal{M}(A) \ar[r]^{\bar{\varphi}} & \ldots \ar[r] & B}.$$
Let $\alpha$ denote the natural automorphism on $B$ coming from this inductive limit structure. Now, the algebra $D$ in the crossed product is given by the inductive limit of building blocks $D_n=\text{C*}(A_{-n},\ldots,A_{-1},A_0,A_1,\ldots, A_n)$ with injective connecting maps given by inclusion. Here, $A_0:=\mu_{\infty,0}(A)\cong C(\prod_{j=1}^\infty S^2,\mathcal{K})$, $A_n:=\alpha^n(\mu_{\infty,0}(A))$ for all $n\in\mathbb{Z}$. The properties of $\varphi$ imply that $A_n\cap A_m=\{0\}$ and $A_nA_m=A_{min\{n,m\}}$. 

The infinite projection $\mu_{\infty,0}(g)$ in $D\rtimes_\alpha\mathbb{Z}$ is given by the image of the trivial projection $g$ in $C(\prod_{j=1}^\infty S^2,\mathcal{K})\cong A_0=D_0$. (The map is given by the composition of the inclusion of $D_0$ into $D$ and the natural inclusion of $D$ into $D\rtimes_\alpha\mathbb{Z}$.) The finite projection is given by the image of the Bott projection, $Q:=\mu_{\infty,0}(p_1)\in D_0\hookrightarrow D\rtimes_\alpha\mathbb{Z}$, where $p_1$ denotes the Bott projection over the first coordinate of $\prod_{j=1}^\infty S^2$. 

We have that $\alpha(\mu_{\infty,0}(p_1))=\mu_{\infty,0}(\varphi(p_1))$. In $\mathcal{M}(A)$, we have that $\varphi(p_1)>q_n, n=1,2,\ldots$ for an infinite sequence of mutually orthogonal projections $q_n$ in $A$. Each $q_n$ is equivalent in $A$ to a Bott projection $p_{\nu(n)}$ with $\nu(n)\in \mathbb{N}$ denoting the coordinate of $\prod_{j=1}^\infty S^2$ over which the Bott projection is taken. (In the notation of \cite{R03} we have $\varphi(p_1)> \sum_{j=-\infty}^0 S_j p_{\nu(j,1)} S_j^*$, so $q_n:=S_{(-n)} p_{\nu(-n,1)} S_{(-n)}^*$.)
 
Setting $z_n:=\langle \mu_{\infty,0}(q_n) \rangle$ (where $\langle a \rangle$ denotes the Cuntz class of $a$), we have that 
$$\sum_{n=1}^\infty z_n= \sum_{n=1}^\infty \langle \mu_{\infty,0}(q_n) \rangle<\langle \mu_{\infty,0}(\varphi(p_1)) \rangle=\langle \alpha(\mu_{\infty,0}(p_1)) \rangle =\langle \mu_{\infty,0}(p_1) \rangle=\langle Q\rangle$$ is finite, and $2\cdot z_n=2\cdot \langle \mu_{\infty,0}(q_n)\rangle =\langle \mu_{\infty,0}( p_{\nu(-n,1)}\oplus p_{\nu(-n,1)})\rangle\geq \langle \mu_{\infty,0}(g)\rangle =\infty,\ n\in\mathbb{N}.$
\end{proof}

The next result provides the relation between the corona factorization property and property {\rm (QQ)} for simple C*-algebras. Proposition \ref{PropQQ} shows that property (QQ) implies the CFP. Example \ref{Ex:CFP-noW} and Example \ref{Ex:stablyNoFunctionals} show that the converse does not hold. However, if we rule out examples like the ones in \ref{Ex:CFP-noW} and \ref{Ex:stablyNoFunctionals} by assuming cancellation of small elements at infinity, then we do get the converse.

\begin{theorem}\label{CFPVsQQAlg}
Let $A$ be a simple $C^*$-algebra.  Then $\Cu(A)$ has property {\rm (QQ)} if and only if $\Cu(A)$ has both the CFP and cancellation of small elements at infinity.  
\end{theorem}

\begin{proof}
Proposition \ref{QQImpliesCAI} and Proposition \ref{PropQQ} show that the 'only if'-direction holds. 

For the converse let us assume cancellation of small elements at infinity to hold. As in the proof of Proposition \ref{CFPVsQQ} we note that all we need to show is that, if there is some $x\in Cu(A)$ with $m\cdot x=\infty$ for some $m$, yet $x\neq\infty$, then there is a sequence $(z_n)_n$ such that $N\cdot \sum_{n=1}^\infty z_n=\infty$ for some $N$, yet $\sum_{n=1}^\infty z_n\neq \infty$. (In Proposition \ref{CFPVsQQ} we saw that this is easy with axiom (O5) in the algebraic case, i.e., in the case that every element in $Cu(A)$ can be written as the supremum of compact elements.)

The proof is divided in cases:

At first, suppose that we have $x\in \Cu(A)$ such that $m\cdot x=\infty$ for some $m>2$, but $(m-1)\cdot x\neq \infty$. Find $a\in (A\otimes \mathcal{K})_+$ of norm 1 with $\langle a \rangle=x$. For given $\alpha< \beta \in\mathbb{R}$, let $f_{\alpha,\beta}$ denote the function from $\mathbb{R}_+$ into itself given by
$$f_{\alpha, \beta}(t)=\left \{\begin{array}{ll} 0 & ,0\leq t\leq \alpha \mbox{ and }t  \geq \beta \\ 1 &, t=(\beta+\alpha)/2    \end{array} \right. \mbox{, and linear elsewhere.} $$
Note that $f_{\alpha,\beta}(a)\ll\infty$ for each $0<\alpha<\beta$. We set $a_n:=f_{1/2^{n},3/ 2^{n}}(a) $ and $z_n:=\langle a_n \rangle$, $n\geq 1$. Then $a_{2n}$ is orthogonal to $a_{2k}$, and $a_{2n+1}$ is orthogonal to $a_{2k+1}$, whenever $k\neq n$. It follows that $\sum_{n=1}^\infty z_{2n}\leq \langle a \rangle =x$, and also $\sum_{n=1}^\infty z_{2n-1}\leq x$, so $\sum_{n=1}^\infty z_{n}\leq 2x$. On the other hand, $\sum_{n=1}^N z_{n}\geq  \langle (a-1/2^N)_+ \rangle$ for each $N$, where $(a-\epsilon)_+=g(a)$ for $g(t)=\max\{0,t-\epsilon\}$. Hence,
$$ \sum_{n=1}^\infty z_{n}\leq 2x<\infty, \mbox{ and }m\cdot \sum_{n=1}^\infty z_{n}\geq m\cdot x=\infty.$$
We found our desired sequence $(z_n)_n$.

In the case that $m=2$, we may try to proceed as before to find the sequence $(z_n)_n$. In this case there exist two possibilities: If we are lucky, the $z_n$'s satisfy $\sum_{n=1}^\infty z_n <\infty$, in which case we are done, just as before. But possibly $\sum_{n=1}^\infty z_n =\infty$, in which case we need to restart to choose the $z_n$'s more carefully. Let us study this second case.

Suppose $2x=\infty$, $x\neq \infty$, and find $a\in (A\otimes \mathcal{K})_+$ of norm 1 with $\langle a \rangle=x$. Set $a_1:=(a-1/2)_+$, and $z_1:=\langle a_1 \rangle \in \Cu(A)_{\ll\infty}$. Let $y_2:=\langle f_{0,3/4}(a)\rangle$. Then $2y_2+2z_1\geq 2x=\infty$. By cancellation of small elements at infinity we must have that $2y_2=\infty$. We can write $y_2=\sup_n \langle f_{1/n,3/4}(a)\rangle $. Hence, $$z_1\ll \infty = 2\cdot y_2=2\cdot \sup_n \langle f_{1/n,3/4}(a)\rangle ,$$ so we can find $0< \delta_2< 1/2$ such that $z_1\leq 2\cdot \langle f_{\delta_2,3/4}(a)\rangle$. We set $a_2:=f_{\delta_2,3/4}(a)$ and $z_2:=\langle  a_2 \rangle$. 

Now find $\delta_2<\gamma_2<1/2$, set $a_3:=f_{\delta_2/2,\gamma_2}(a)$ and set $z_3:=\langle a_3 \rangle \in Cu(A)_{\ll\infty}$.

Similar to the previous step, we set $y_4:=\langle f_{0,3\delta_2/4}(a)\rangle$ and get 
$$z_3\ll \infty = 2\cdot y_4=2\cdot \sup_n \langle f_{1/n,3\delta_2/4}(a)\rangle.$$ Thus, proceeding inductively, we get a sequence $(a_n)_n$ of positive elements in $A\otimes \mathcal{K}$ and a sequence $(z_n)_n$ in $\Cu(A)$, such that:
\begin{itemize}
\item[(1)] $z_n=\langle a_n\rangle\ll\infty$ for all $n$.
\item[(2)] $a_n\leq a$ for all $n$.
\item[(3)] For all $n\neq k$, $a_{2n}$ is orthogonal to $a_{2k}$, and $a_{2n+1}$ is orthogonal to $a_{2k+1} $.
\item[(4)] $z_{2n-1}\leq 2 z_n$ for all $n$.
\item[(5)] $\sum_{n=1}^\infty z_{2n}\leq \langle a \rangle=x$ and $\sum_{n=1}^\infty z_{2n+1}\leq x$.
\item[(6)] $\sum_{n=1}^\infty z_{n}\geq \langle a \rangle=x$.
\end{itemize}

Recall that by assumption we have $\sum z_n=\infty$, and that by (4) it follows that $\sum_{n=1}^\infty z_{2n-1}\leq 2\cdot \sum_{n=1}^\infty z_{2n}$. Therefore,
$$3\cdot \sum_{n=1}^\infty z_{2n}\geq \sum_{n=1}^\infty z_{n}=\infty \mbox{ , while } \sum_{n=1}^\infty z_{2n}\leq x \neq \infty.$$ 
We found the desired sequence with $(z_{2n})_n$.
\end{proof}

We conclude this paper with an overview of our results on comparison properties for the Cuntz semigroup of a C*-algebra, together with a list of interesting open questions that naturally arise from our studies.

\begin{theorem}\label{thm:end}
Let $A$ be a simple C*-algebra. Then we have the following diagram of relations for comparison properties of the $\Cu$-semigroup $\Cu(A)$:
% $$\xymatrix{
% \omega-comparison \ar@{<=>}[r] \ar@{=>}[d]& \beta-comparison \ar@{<=>}[r]  \ar@{=>}[d]  & [\lambda(y)=\infty \mbox{ for all functionals } \lambda\Leftrightarrow  y=\infty] \ar@{<=>}[d] \\ {\rm (QQ)} \ar@{<=>}[r]&  CFP\mbox{ and canc. at }\infty & [\nexists \mbox{ faithful state }f\in\mathcal{S}(\Cu(A),y)\Leftrightarrow  y=\infty  ] }$$
% \end{itemize}
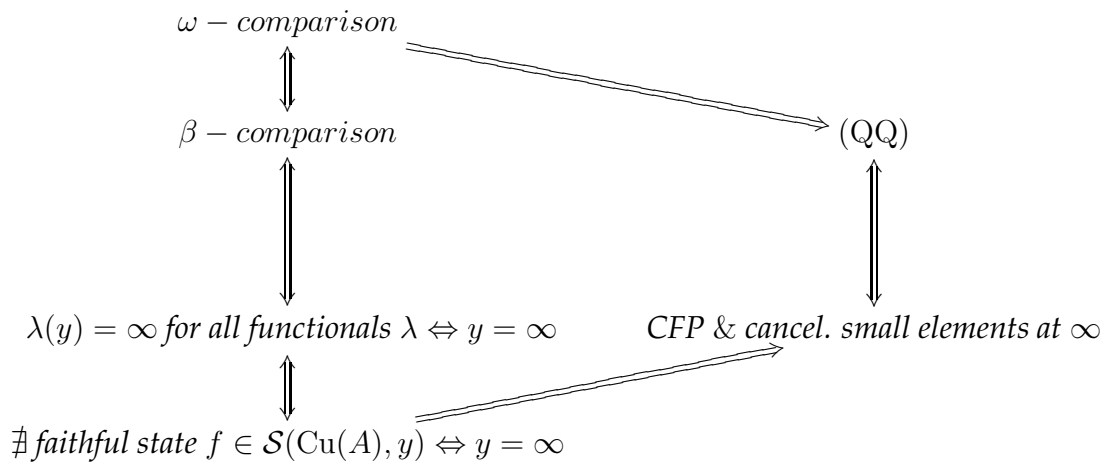
$$\xymatrix{
\omega-comparison \ar@{<=>}[d] \ar@{=>}[rd]& \\ \beta-comparison \ar@{<=>}[dd]&{\rm (QQ)} \ar@{<=>}[dd]  \\&  \\\ \lambda(y)=\infty \mbox{ for all functionals } \lambda\Leftrightarrow  y=\infty\ar@{<=>}[d]&\mbox{CFP  $\&$ cancel. small elements at }\infty\\ \nexists \mbox{ faithful state }f\in\mathcal{S}(\Cu(A),y)\Leftrightarrow  y=\infty  \ar@{=>}[ru] & }$$
\end{theorem}

\begin{question}\label{QN}
\begin{itemize}
\item Is there any simple C*-algebra $A$ such that $\Cu(A)=[0,1]\cup\{\infty\}$?\newline Or any
 stably finite C*-algebra such that $\Cu(A)=\{(0,0)\} \cup ((0,1]\cup\{\infty\})\times(0, \infty]?$
\item Does $\Cu(A)$ have cancellation of small elements at infinity for any simple stably finite C*-algebra?
\item Is CFP plus cancellation of small elements at infinity equivalent to $\omega$-comparison for any $\Cu$-semigroup? What about for the $\Cu$-semigroup arising from a C*-algebra? 
\item Is CFP is equivalent to $\omega$-comparison for the Cuntz semigroup arising from a C*-algebra?
\end{itemize}
\end{question}

\end{document}